\documentclass[a4paper,12pt,reqno]{amsart}
\usepackage{amsmath,amsfonts,amssymb,amsthm,enumerate,multicol}
\usepackage{mathrsfs}
\usepackage{tikz,graphicx}
\usepackage{hyperref}
\usepackage{caption,enumitem,wasysym}
\usepackage{cleveref}
\usepackage{epstopdf}
\usepackage{float,wrapfig}
\usepackage[labelsep=space]{caption}
\usepackage[font=footnotesize]{caption}
\usepackage{xcolor}
\DeclareFontShape{OT1}{cmr}{bx}{sc}{<->cmcsc10}{}
\makeatletter
\DeclareFontFamily{U}{wasy}{}
\DeclareFontShape{U}{wasy}{m}{n}{<->wasy10}{}
\DeclareFontShape{U}{wasy}{b}{n}{<->wasy10}{}
\makeatother
\usepackage{cases}
\usepackage[english]{babel}
\usepackage[autostyle]{csquotes}
\usepackage{cite}
\theoremstyle{plain}
\newtheorem{thm}{Theorem}[section]

\newtheorem{lem}[thm]{Lemma}
\newtheorem{prop}[thm]{Proposition}

\newtheorem{defn}[thm]{Definition}
\newtheorem{rem}[thm]{Remark}

\newtheorem{assum}[thm]{Assumptions}

\usepackage{mathtools}
\usepackage{amsmath,amsfonts,amssymb,amsthm,enumerate,multicol}
\usepackage{tikz}
\usepackage{float}
\usepackage{autobreak}

\allowdisplaybreaks

\numberwithin{equation}{section}

\begin{document} 
	\begin{center}
		\Large{\textbf{Existence for the Discrete Nonlinear Fragmentation Equation with Degenerate Diffusion}}
	\end{center} 
	\medskip
	\medskip
	\centerline{${\text{ ${\text{Saumyajit Das$^{\dagger*}$}}$ and  Ram Gopal Jaiswal$^{\dagger*\ddagger*}$}}$}\let\thefootnote\relax\footnotetext{$^{*}$Corresponding author.  \newline{\it{${}$ \hspace{.3cm} Email address: }}saumyajit.math.das@gmail.com, rgopaljaiswal@iitr.ac.in}
	\medskip
	{\footnotesize

		\centerline{ ${}^{}$  $\dagger$ Department of Mathematics, Harish-Chandra Research Institute,}
		\centerline{ A CI of Homi
			Bhabha National Institute, Chhatnag Road,}
		\centerline{Jhunsi, Prayagraj 211019, India.}
		\centerline{ ${}^{}$  $\ddagger$ Department of Mathematics, Indian Institute of Technology Bombay,}
		\centerline{Powai, Mumbai, Maharashtra 400076, India}
		
	}
	\bigskip
	\begin{quote}
		{\small {\em \bf Abstract.} 
			A mathematical model for the discrete nonlinear fragmentation (collision-induced breakage) equation with diffusion is studied. The existence of global weak solutions is established in arbitrary spatial dimensions without assuming a strictly positive lower bound on the diffusion coefficients, extending previous results that were restricted to one-dimensional domains and relied on uniformly positive diffusion. The analysis is carried out under boundedness assumptions on the collision and breakage kernels. The proof is based on the construction of a suitable regularized system, combined with weak $L^2$ a priori estimates and compactness arguments in $L^1$, which allow the passage to the limit in the nonlinear fragmentation operator.}
	\end{quote}
	\vspace{0.5cm}
	\section{\textbf{Introduction}}\label{sec:intro}
	
	A system of particles is considered in which collisions between two particles result in fragmentation into daughter particles with a possible transfer of mass. These particles can also diffuse in space with a diffusion coefficient that depends on their size. This mechanism is described by the nonlinear fragmentation equation (collision-induced breakage equation) with diffusion. Let $\Omega \subset \mathbb{R}^N$ be a smooth bounded domain. For $i \ge 1$, let $f_i(t,y)$ denote the density of particles of size $i$ at position $y \in \Omega$ and time $t \ge 0$. For $T>0$, set $\Omega_T = (0,T)\times\Omega$. Then, the nonlinear fragmentation equation with diffusion reads
	\begin{equation}
		\left\{
		\begin{aligned}
			& \partial_t f_i - d_i\, \Delta f_i = Q_i(f), && \text{in } (0,T)\times \Omega, \\
			& \nabla{f_i}\cdot \nu = 0, && \text{on } (0,T)\times \partial\Omega, \\
			& f_i(0,x) = f_i^{\mathrm{in}}(x), && \mbox{in}\  \Omega,
		\end{aligned}
		\right. \label{NFE}
	\end{equation}
	where
	\begin{align} 
		Q_i(f)
		= 
		\frac12 \sum_{j=i+1}^{\infty}\sum_{k=1}^{j-1}
		b_{j-k,k}^i\, a_{j-k,k}\, f_{j-k} f_k
		\;-\;
		\sum_{j=1}^{\infty} a_{i,j}\, f_i f_j. \label{NFE Operator}
	\end{align}
	
	In equations \eqref{NFE}–\eqref{NFE Operator}, $d_i>0$ denotes the (constant) diffusion coefficient of particles of size $i$, and $\nu$ represents the unit outward normal vector at a point $x\in\partial\Omega$. The collision kernel denoted by $a_{i,j}$ accounts for the rate of collision between particles of size $i$ and $j$. For $i\in\mathbb{N}$, the initial densities $f_i^{\rm{in}}$ are taken to be nonnegative and smooth. We assume that the total initial mass is finite,
	\begin{align}
		\sum_{i=1}^{\infty} i \, \| f_i^{\rm in} \|_{L^1(\Omega)} < +\infty, \label{IC mass finite}
	\end{align}
	and that the initial densities are uniformly bounded, i.e.,
	\begin{align}
		\sup_{i \in \mathbb{N}} \| f_i^{\rm in} \|_{L^\infty(\Omega)} < +\infty. \label{IC uniform boundedness}
	\end{align}
	The breakage kernel $b^k_{i,j}$ gives an average number of particles of size $k$ due to collision between particles of size $i$ and $j$. 
	
	The collision kernel is assumed to be symmetric, that is,
	$a_{i,j} = a_{j,i}$ for all $i,j \ge 1$, and nonnegative.
	The breakage kernel is also symmetric with respect to the indices $i$ and $j$,
	namely $b_{i,j}^k = b_{j,i}^k \ge 0$ for all $i,j,k \ge 1$.
	Moreover, the breakage kernel satisfies the mass conservation property in each collision given by
	\begin{align}
		\sum_{k=1}^{i+j-1} k\, b_{i,j}^k = i+j, \qquad \text{for all } i,j \ge 1. \label{local mass conservation}
	\end{align}
	In addition, no particles larger than the total colliding mass are produced, that is,
	\begin{align}
		b_{i,j}^k = 0, \qquad \text{for } k \ge i+j. \label{zero value of b}
	\end{align}

	However, particles of size larger than one of the colliding particles may be
	formed as a result of mass transfer.
	For example, when particles of sizes $i$ and $j$ collide, it is possible to form
	particles of sizes $i + j/3$ and $2j/3$ by transferring a mass
	$j/3$ from the particle of size $j$ to the particle of size $i$ after
	collision.
	
	However, it is worth mentioning an important subcase in which mass transfer
	between colliding particles is excluded.
	In this setting, both particles break into smaller fragments without any transfer
	of mass between them.
	This collision-induced breakage mechanism was first studied by Cheng and Redner \cite{cheng1988, cheng1990} and is commonly referred to as the nonlinear fragmentation equation without mass transfer. In this case, the breakage kernel admits the representation
	\[
	b_{i,j}^k
	=
	\beta_{i,j}^k \, \mathbf{1}_{[k,\infty)}(i)
	+
	\beta_{j,i}^k \, \mathbf{1}_{[k,\infty)}(j),
	\]
	where $\beta_{i,j}^k$, for $1 \le k \le i-1$, denotes the average number of
	particles of size $k$ produced from a particle of size $i$ after a collision with
	a particle of size $j$.
	The coefficients $\beta_{i,j}^k$ satisfy the mass conservation condition
	\[
	\sum_{k=1}^{i-1} k\, \beta_{i,j}^k = i, \qquad \text{for all } i,j \ge 1,
	\]
	together with
	\[
	\beta_{i,j}^k = 0, \qquad \text{for } k \ge i.
	\]
	
	The homogeneous nonlinear fragmentation equation, that is, the nonlinear
	fragmentation equation without diffusion, has attracted considerable attention.
	Results concerning existence, uniqueness, and large-time behavior including the
	existence of self-similar solutions, have been studied by several authors; see, for
	example, \cite{ali2024well, ali2024note, ali2024, ali2025, ernst2007, zheng2005,
		laurencot2001, kostoglou2000, GL22021, JA2025, GJL2024, barik2020, barik2020global,
		barik2021, barik2021existence, GL12021}.
	The existence of solutions to inhomogeneous nonlinear fragmentation equations with
	scattering, that is, models in which the maximum cluster size is finite, has also
	been investigated by several authors and can be found in
	\cite{amann2005local, walker2004discrete, walker2005new}.
	The first results concerning the existence of solutions to inhomogeneous nonlinear
	fragmentation equations with arbitrarily large particle sizes in one dimension were
	obtained by Canizo, Desvillettes, and Fellner \cite{canizo2010regularity}. Their
	existence result relies on an $L^2$ estimate of the mass, which requires the diffusion
	coefficients to be uniformly bounded from below and above by positive constants.
	This framework, however, does not cover the physically relevant case in which the
	diffusion coefficients are degenerate, that is,
	
	\[
	\inf_{i \ge 1} d_i = 0.
	\]
	
	It is worth emphasizing that, in contrast to the nonlinear fragmentation equation,
	the linear fragmentation equation with coagulation with diffusion has been extensively studied in the
	literature \cite{wrzosek2002, laurenccot2002, laurencot2002continuous}. Although the coagulation operator is itself nonlinear, when it is
	coupled with linear fragmentation, the specific structure of the resulting
	coagulation--fragmentation system allows to derive a priori
	$L^\infty$ bounds by an induction argument on the cluster size. Such uniform bounds
	play a crucial role in establishing the global-in-time existence of solutions. However, this induction-based approach does not seem to be applicable to the case
	of nonlinear fragmentation, due to the fundamentally different structure of the
	nonlinear fragmentation operator. This difficulty was already highlighted in \cite{canizo2010regularity}, whose approach relies on an $L^2$ estimate on mass and provides an alternative method for proving existence results for
	nonlinear fragmentation models, as well as for coagulation--fragmentation systems
	with non-degenerate diffusion coefficients. We also note that degenerate diffusion coefficients have been considered for the
	coagulation--fragmentation equation with diffusion in \cite{canizo2010absence}. To the best of our knowledge, however, the case of degenerate diffusion
	has not been addressed for the nonlinear fragmentation equation. In particular,
	the lack of a priori $L^\infty$ bounds, due to the failure of the induction
	argument, constitutes a major additional difficulty in the nonlinear
	fragmentation setting.

	\medskip
	The main objective of the present paper is to address this physically relevant
	case by considering diffusion coefficients of the form
	\[
	d_i = \frac{1}{i^\alpha}, \qquad \alpha \ge 0,
	\]
	for nonlinear fragmentation equation with diffusion along with bounded breakage and collision kernels. More precisely, the analysis relies on the following assumptions.
	
	\begin{assum}\label{weaker assumption}
		\begin{enumerate}
			\item[(A1)] The collision kernel $a_{i,j}=a_{j,i}\ge0$ satisfies
			\[
			a_{i,j} = (ij)^{-\lambda}, \qquad \lambda \ge 4.
			\]
			
			\item[(A2)] The breakage function $b_{i,j^k}$ satisfies
			\[
			b_{i,j}^k = \frac{2}{i+j-1}, \qquad 1 \le k \le i+j-1.
			\]
			
			\item[(A3)] The diffusion coefficients satisfy
			\[
			d_i = \frac{1}{i^\alpha}, \qquad \alpha \in [0,1].
			\]
		\end{enumerate}
	\end{assum}
	
	It is worth mentioning that the analysis also applies to a broader class of kernels. In particular, one may choose any diffusion coefficients $d_i > 0$
	that are uniformly bounded above, provided that the collision kernel and the
	breakage function satisfy the following conditions.
	
	\begin{assum}\label{strong assumption}
		The following assumptions apply to the collision kernel $a_{i,j}$, the
		breakage function $b_{i,j}^k$, and the diffusion coefficients $d_i$
		
		\begin{enumerate}
			\item[(A1)] The collision kernel $a_{i,j}=a_{j,i}\ge0$ satisfies the summability condition
			\[
			\sum_{i,j \ge 1} (i+j)\, a_{i,j} < \infty.
			\]
			
			\item[(A2)] There exists a constant $B > 0$ such that
			\[
			b_{i,j}^k \le B, \qquad 1 \le k \le i+j-1, \quad i,j \ge 1.
			\]
			
			\item[(A3)] The diffusion coefficients satisfy
			\[
			\lim_{i \to \infty} d_i = 0.
			\]
			
			\item[(A4)] The following summability condition holds
			\[
			\sum_{j=1}^{\infty}
			\sum_{k=1}^{\infty}
			\sum_{i=1}^{j+k-1}
			\frac{\sqrt{\, b^{\,i}_{\,j,k}\, a_{\,j,k}}}{\sqrt{k j\, d_i d_j}}
			\;+\;
			\sum_{i=1}^{\infty}\sum_{j=1}^{\infty}
			\frac{\sqrt{a_{i j}}}{\sqrt{i j\, d_i d_j}}
			< \infty.
			\]
		\end{enumerate}
	\end{assum}

	Next, we define the notion of solution for the nonlinear fragmentation equation \eqref{NFE}.
	\begin{defn}[Weak solution]\label{definition weak solution}
		Let $T>0$ and $i\ge1$. A function
		\[
		f_i \in L^1\!\left((0,T); W^{1,1}(\Omega)\right)
		\]
		is called a weak solution to the system \eqref{NFE} with nonegative initial datum
		$f_i^{\mathrm{in}}\in L^1(\Omega)$ if, for all 
		\[
		\psi \in C_c^\infty([0,T)\times\Omega),
		\]
		the following inequality holds
		\[
		\int_\Omega
		\psi(0,x)\, f_i^{\mathrm{in}}(x)\, dx
		+
		\int_0^T \int_\Omega
		\left(-\partial_t \psi\right) f_i
		\,dx\,dt
		+
		\int_0^T \int_\Omega
		\nabla \psi \cdot \nabla f_i
		\,dx\,dt
		=
		\int_0^T \int_\Omega
		\psi \, Q_{i}(f)
		\,dx\,dt,
		\]
	\end{defn}
	Similarly, we can define the supersolution and the subsolution according to the greater than and less than symbol. To define the supersolution and subsolution, we choose the test function $\psi$ to be nonnegative.
	Next we state the main result of our article:
	\begin{thm}\label{existence result}
		Let the coagulation kernel $a_{i,j}$, the breakage kernel $b_{i,j}^k$, and the diffusion coefficients $d_i$ satisfy either Assumption~\ref{weaker assumption} or Assumption~\ref{strong assumption} for all $i,j,k \in \mathbb{N}$. Suppose that the initial data satisfy
		\[
		\sum_{i=1}^{\infty} \frac{1}{\sqrt{d_i}} \, \|f_i^{\rm{in}}\|_{L^1(\Omega)}^{\frac12} < +\infty \ \mbox{and} \ \left\|\sum_{i=1}^{\infty}if_i^{\rm{in}}\right\|_{L^2(\Omega)}<+\infty.
		\]
		Then the system~\eqref{NFE} admits a global-in-time nonnegative weak solution $\{f_i\}_{i\in\mathbb{N}}$ such that, for every $i \in \mathbb{N}$,
		\[
		f_i \in L^1\!\left((0,T); W^{1,1}(\Omega)\right)
		\cap L^2\!\left((0,T)\times\Omega\right).
		\]
	\end{thm}
	\begin{rem}
		The second condition on the initial data allows us to obtain solutions in the space
		$L^2\big((0,T)\times\Omega\big)$, whereas the first condition is of a more technical nature.
		The latter is slightly stronger than the finiteness of the total initial mass.
		Moreover, it reflects the fact that, as the particle size increases, the initial mass per unit size becomes sufficiently small; mathematically, this is captured by the first assumption on the initial condition. This assumption is crucial for establishing the existence of solutions in this article (see equation \eqref{use_small_initial_mas} in  Proposition \ref{control by time parameter} ).
	\end{rem}
	
	We note that the problem involves two main difficulties: first, the number of unknowns is infinite; second, the source terms associated with \eqref{NFE} exhibit quadratic nonlinearities. We further observe that the source terms satisfy a mass-conservation structure, namely,
	\begin{align}\label{mass conservation NCF}
		\sum_{i=1}^{\infty} i \, Q_i = 0 .
	\end{align}
	Such a structure plays an important role in the study of reaction-diffusion systems with a finite number of unknowns. Studies of reaction-diffusion equations with quadratic nonlinearities can be found in \cite{pierre2003weak, Pierre2010}. For a comprehensive treatment of reaction-diffusion systems, we refer the reader to \cite{canizo2014improved, desvillettes2015duality, desvillettes2007global, fitzgibbon2021reaction, fellner2021uniform, das2025existence}. We emphasize that our approach is motivated by the work of Michel Pierre \cite{Pierre2010}. 
	
	In the present setting, the analysis involves additional complexity due to the presence of infinitely many unknowns and infinite summations. Another important structure satisfied by the source terms is quasipositivity; see \eqref{quasipositive}. Namely, for the source term $Q_i$ in \eqref{NFE}, if we set $f_i = 0$ while keeping all other unknowns nonnegative, then $Q_i$ remains nonnegative.
	
	We construct an approximation \eqref{eq:TRNFE}, consisting of a reaction-diffusion system with finitely many unknowns, while preserving both the mass-conservation and quasipositivity structures \eqref{quasipositive} of the source terms. To guarantee the existence of global-in-time solutions, we further regularize the source terms. Consequently, our approximation involves both a truncation of the number of unknowns and a regularization of the source terms.
	
	Thanks to the following theorem, the quasipositivity structure of the source terms ensures that the solutions remain nonnegative, provided the initial data are nonnegative.
	\begin{thm}[Positivity of solutions, \cite{Pierre2010}]\label{positivity of solution}
		Let $\mu_i>0$ for $1\le i\le m$. Let $u_i:(0,T)\times\Omega\to\mathbb{R}$ be a classical solution of the system
		\[
		\left\{
		\begin{aligned}
			\partial_t u_i - \mu_i \Delta u_i &= f_i(u_1,\dots,u_m) && \text{in } (0,T)\times\Omega,\\
			\nabla u_i\cdot\nu &= 0 && \text{on } (0,T)\times\partial\Omega,\\
			u_i(0,x) &\ge 0 && \text{in } \Omega.
		\end{aligned}
		\right.
		\]
		for $1\le i\le m$. Assume that $f=(f_1,\dots,f_m):\mathbb{R}^m\to\mathbb{R}^m$ is quasipositive, that is,
		\begin{align}\label{quasipositive}
			f_i(r_1,\dots,r_{i-1},0,r_{i+1},\dots,r_m) \ge 0,
			\quad \forall (r_1,\dots,r_m)\in [0,+\infty)^m.
		\end{align}
		Then the solution remains nonnegative, namely $u_i\ge 0$ for all $i=1,\dots,m$.
	\end{thm}
	We will show that the sequence of solutions to the truncated and regularized system converges to a distributional solution of the nonlinear fragmentation equation \eqref{NFE}. To extract a convergent subsequence from the approximate solutions, we employ the $L^1$ compactness theorem developed in \cite{Pierre2010}. Furthermore, thanks to the quasipositivity structure \eqref{quasipositive} of the source terms in the approximate system, the solutions remain nonnegative.
	
	It is evident that, to ensure the source term belongs to $L^1((0,T)\times\Omega)$, we require a particular weighted sum of $L^2((0,T)\times\Omega)$ norm estimates on the solutions of the approximate system, since the source terms exhibits quadratic growth and involve an summation structure. Such an estimate for reaction-diffusion systems with finitely many unknowns can be found in \cite{desvillettes2007global}. The result is as follows:
	\begin{thm}[\cite{desvillettes2007global}]\label{finite L^2 estimate}
		Let $d_i\geq 0$. Let $\mathcal{F}_i\geq 0$ satisfies: for $1\leq i\leq m$
		\begin{equation*}
			\left \{
			\begin{aligned}
				\partial_t \left( \sum_{i=1}^m \mathcal{F}_i\right)- \Delta \left( \sum_{i=1}^m d_i\mathcal{F}_i\right)=&0 \qquad && \mbox{in}\ (0,T)\times\Omega\\
				\nabla \mathcal{F}_i\cdot \nu=& 0 \qquad && \mbox{on} \ (0,T)\times\partial\Omega\\
				\mathcal{F}_i(0,x)=& \mathcal{F}_{i}^{\rm{in}} \qquad && \mbox{in}\ \Omega.
			\end{aligned}
			\right .
		\end{equation*}
		Let $\mathcal{F}_i^{\rm in} \in L^2((0,T)\times\Omega)$ for $i = 1, \dots, m$.
		Then
		\[
		\int_{0}^T \int_{\Omega} \sum\limits_{i=1}^m \mathcal{F}_i \times \sum\limits_{i=1}^m d_i\mathcal{F}_i \, dx \, dt \leq C(\Omega,T) \max\limits_{1\leq i\leq m}\{d_i\} \times \sum_{i=1}^m \| \mathcal{F}_i^{\rm{in}}\|_{L^2(\Omega)},
		\]
		where the constant $C(\Omega,T)$ depends on the domain and the terminal time $T$.
	\end{thm}
	We will obtain a similar weighted sum of $L^2$ norm estimate for the infinite system. 
	For the simpler coagulation model, such an estimate can be found in \cite{canizo2010absence}. Following the ideas in \cite{Pierre2010}, by employing such weighted sum of $L^2$ norm estimates, we will obtain a weak supersolution to the nonlinear fragmentation system~\eqref{NFE}. Together with the mass-conservation structure, this allows us to construct a distributional solution to the nonlinear fragmentation system~\eqref{NFE}.

	\subsection{Plan of the paper}
	
	We divide this article into several sections. In the first section~\ref{TRAS}, we
	introduce a truncated and regularized reaction-diffusion system with finitely many
	unknowns, where the source terms preserve the quasipositive and mass-conservative
	structure. In this section, we analyze the existence of global-in-time nonnegative
	classical solutions and derive regularity estimates for the solutions of the truncated
	and regularized reaction-diffusion system. Here the truncation is in the number of unknowns `$n$' and the regularization `$\varepsilon$' is in the source terms.
	\newline
	Our goal is to recover a solution of the nonlinear fragmentation system \eqref{NFE} 
	as a limiting function of the solutions of the truncated and regularized system. In 
	the section~\ref{RS}, we pass to the limit with respect to the truncation 
	index $n$, which yields a smooth, global-in-time, nonnegative solution of the 
	regularized nonlinear fragmentation system.
	\newline
	In the section~\ref{UECRS}, we analyze the compactness of the solutions of the
	regularized nonlinear fragmentation system. Here, the $L^1((0,T); W^{1,1}(\Omega))$
	compactness theorem, as developed in \cite{Pierre2010}, plays a crucial role.
	\newline
	Passing to the limit with respect to the regularization parameter $\varepsilon$ is more
	delicate due to the presence of infinite nonlinear sums in the source terms. In the
	next section~\ref{TPCRT}, we devote ourselves to analyzing the asymptotic behavior as
	$\varepsilon \to 0$. To handle the infinite sums, we introduce an additional truncation
	in the range of the solutions of the regularized nonlinear fragmentation system. More
	precisely, we estimate the asymptotic behavior with respect to $\varepsilon$ on suitable
	level sets of the regularized solutions. We mostly follow the approach described in \cite{Pierre2010}.
	\newline
	This procedure introduces further difficulties in the time variable, since derivatives
	with respect to both time and space variables are not well behaved near the initial time $t = 0$.
	To overcome this issue, we study the solutions on the time interval $(\delta, T)$ and
	then analyze the limit as $\delta \to 0$. On each level set, we are able to pass to the
	limit in the regularized equation, more precisely, in its weak formulation, with respect to both
	parameters $\delta$ and $\varepsilon$. However, this comes at the cost of obtaining an
	inequality of greater-than-or-equal-to type in the final limit.   
	\newline
	The section~\ref{EWSNFE} is devoted to the construction of a global-in-time weak
	nonnegative supersolution to the nonlinear fragmentation system \eqref{NFE}. Thanks to the
	greater-than-or-equal-to type estimate obtained in the previous section, passing to the
	limit with respect to the level-set parameter yields a global-in-time weak nonnegative
	supersolution.
	\newline
	In the final section~\ref{final_section}, we construct a global-in-time weak
	nonnegative solution to the nonlinear fragmentation system \eqref{NFE}. The
	mass-conservative structure of the truncated regularized system leads us to the existence of a global-in-time weak nonnegative solution.

	\subsection{Notation}
	\begin{itemize}
		\item [$\bullet$] We use the notation $a \lesssim b$ to denote that there exists a constant $\texttt{C} > 0$
		such that $a \le \texttt{C}\, b$. Throughout this article, such constants are independent of all indices and regularization parameters.
	\end{itemize}
	
	Throughout the following sections, we assume that the coefficients
	$a_{i,j}$, $b_{i,j}^k$, and $d_i$ satisfy either Assumption~\ref{weaker assumption}
	or Assumption~\ref{strong assumption}. Moreover, the breakage kernel
	$b_{i,j}^k$ satisfies \eqref{local mass conservation} and
	\eqref{zero value of b}. We further assume that the initial condition
	$f^{\mathrm{in}}$ satisfies \eqref{IC mass finite} and
	\eqref{IC uniform boundedness}.

	\section{\textbf{Truncated regularized approximate system}}\label{TRAS}
	
	Let $\Omega \subset \mathbb{R}^N$ be a smooth bounded domain and let $T>0$.
	For $n \ge 3$ and $\varepsilon>0$, we consider the truncated regularized
	system: for $i=1,\dots,n$
	\begin{equation}
		\label{eq:TRNFE}
		\left\{
		\begin{aligned}
			& \partial_t f_{i,\varepsilon}^n - d_i \Delta f_{i,\varepsilon}^n
			= Q_{i,\varepsilon}^n(f_{\varepsilon}^n)
			&& \text{in } (0,T)\times\Omega \\
			& \nabla f_{i,\varepsilon}^n \cdot \nu = 0
			&& \text{on } (0,T)\times\partial\Omega \\
			& f_{i,\varepsilon}^n(0,x) = f_i^{\mathrm{in}}(x)
			&& \text{in}\ \Omega,
		\end{aligned}
		\right.
	\end{equation}
	where
	\begin{align}
		\label{eq:TRNFEoperator}
		Q_{i,\varepsilon}^n(f_{\varepsilon}^n)
		= \frac{
			\frac{1}{2}\displaystyle\sum_{j=i+1}^{n}\sum_{k=1}^{j-1}
			b_{j-k,k}^i a_{j-k,k}\,
			f_{j-k,\varepsilon}^n f_{k,\varepsilon}^n
			-
			\displaystyle\sum_{j=1}^{n-i}
			a_{i,j}\, f_{i,\varepsilon}^n f_{j,\varepsilon}^n
		}{
			1 + \varepsilon\displaystyle\sum_{j=1}^{n}
			c_j (f_{j,\varepsilon}^n)^2
		}.
	\end{align}
	Here $\varepsilon\in(0,1)$ is an arbitrary constant. The coefficients $a_{i,j}$, $b_{i,j}^k$, and $c_j$ are assumed to be nonnegative and satisfy \ref{strong assumption}. For each $j \in \mathbb{N}$, we define the constant
	\[
	c_j := \sum_{i=1}^{\infty} a_{i,j},
	\]
	so that
	\[
	\sum_{j=1}^{\infty} c_j
	= \sum_{j=1}^{\infty}\sum_{i=1}^{\infty} a_{i,j}
	< +\infty.
	\]
	\begin{prop}[Existence and regularity for the truncated system]
		\label{prop:existence_truncated}
		For every $\varepsilon>0$ and $n \ge 3$, the truncated regularized system \eqref{eq:TRNFE}--\eqref{eq:TRNFEoperator}
		admits a unique global-in-time nonnegative strong solution
		\[
		f_{\varepsilon}^n = \{f_{i,\varepsilon}^n\}_{i=1}^n
		\quad \text{on } (0,T)\times\Omega.
		\]
		Furthermore, for every $0<\delta<T$ and every $1<p<\infty$, 
		\[
		f_{i,\varepsilon}^n \in
		W^{1,p}\big((\delta,T);L^p(\Omega)\big)
		\cap
		L^p\big((\delta,T);W^{2,p}(\Omega)\big),
		\qquad i\in \{1,\dots,n\},
		\]
		and
		\[
		f_{i,\varepsilon}^n \in C^\infty\big((0,T)\times\Omega\big).
		\]
		
		Moreover, the following regularity estimates hold
		\begin{align} \label{eq: SORE}
			\| f_{i,\varepsilon}^n\|_{W^{1,p}((\delta,T); L^p(\Omega))}+ \| f_{i,\varepsilon}^n\|_{L^p((\delta,T);W^{2,p}(\Omega))} \lesssim   \| Q_{i,\varepsilon}^n\|_{L^p((0,T)\times\Omega)}+ \| f_{i}^{\rm{in}}\|_{L^p(\Omega)} \lesssim  \frac{1}{\varepsilon},
		\end{align}
		and
		\[
		\| f_{i,\varepsilon}^n \|_{L^\infty((0,T)\times\Omega)}
		\lesssim \frac{1}{\varepsilon}+ \| f_i^{\rm{in}}\|_{L^{\infty}(\Omega)} \lesssim \frac{1}{\varepsilon},
		\qquad i\in \{1,\dots,n\}.
		\]
		Furthermore,
		\[
		\int_\Omega \sum_{i=1}^{n} i f_{i,\varepsilon}^n(t,x)\,dx
		=
		\int_\Omega \sum_{i=1}^{n} i f_i^{\mathrm{in}}(x)\,dx,
		\qquad \forall\, t\in[0,T].
		\]
	\end{prop}
	\begin{proof}
		Note that here the source term is the function $\displaystyle{Q_{i,\varepsilon}^n: \mathbb{R}^n\to\mathbb{R}}$, defined as follows
		\[
		Q_{i,\varepsilon}^n(f_{\varepsilon}^n)
		= \frac{
			\frac{1}{2}\displaystyle\sum_{j=i+1}^{n}\sum_{k=1}^{j-1}
			b_{j-k,k}^i a_{j-k,k}\,
			f_{j-k,\varepsilon}^n f_{k,\varepsilon}^n
			-
			\displaystyle\sum_{j=1}^{n-i}
			a_{i,j}\, f_{i,\varepsilon}^n f_{j,\varepsilon}^n
		}{
			1 + \varepsilon\displaystyle\sum_{j=1}^{n}
			c_j (f_{j,\varepsilon}^n)^2
		}
		\]
		Here $\varepsilon$ plays the role of a regularization parameter. We find the expression for the first partial derivatives. We divide the analysis of the differentiation with respect to the `$l$'-th coordinate into three cases: for $1\leq l\leq n, \, l\neq i, \, l\leq n-i$
		\begin{align*}
			\frac{\partial}{\partial x_l}Q_{i,\varepsilon}^n\big((x_1,\cdots,x_n)\big) = &\frac{\frac12 \sum\limits_{j=l+1}^n b^i_{l,j-l}a_{l,j-l} x_{j-l}+\frac12 \sum\limits_{j=l+1}^n b^i_{j-l,l}a_{j-l,l} x_{j-l}-a_{i,l}x_i}{1+\varepsilon\sum\limits_{j=1}^n c_j (x_j)^2}\\
			+ & \frac{-2\varepsilon c_lx_l}{1+\varepsilon\sum\limits_{j=1}^n c_j (x_j)^2} \times \frac{
				\frac{1}{2}\displaystyle\sum_{j=i+1}^{n}\sum_{k=1}^{j-1}
				b_{j-k,k}^i a_{j-k,k}\,
				x_{j-k} x_k
				-
				\displaystyle\sum_{j=1}^{n-i}
				a_{i,j}\, x_ix_j
			}{
				1 + \varepsilon\displaystyle\sum_{j=1}^{n}
				c_j (x_j)^2},\\
		\end{align*}
		for $1\leq l\leq n, \, l\neq i,\, l> n-i$:
		\begin{align*}
			\frac{\partial}{\partial x_l}Q_{i,\varepsilon}^n\big((x_1,\cdots,x_n)\big) = &\frac{\frac12 \sum\limits_{j=l+1}^n b^i_{l,j-l}a_{l,j-l} x_{j-l}+\frac12 \sum\limits_{j=l+1}^n b^i_{j-l,l}a_{j-l,l} x_{j-l}}{1+\varepsilon\sum\limits_{j=1}^n c_j (x_j)^2}\\
			+ & \frac{-2\varepsilon c_lx_l}{1+\varepsilon\sum\limits_{j=1}^n c_j (x_j)^2} \times \frac{
				\frac{1}{2}\displaystyle\sum_{j=i+1}^{n}\sum_{k=1}^{j-1}
				b_{j-k,k}^i a_{j-k,k}\,
				x_{j-k} x_k
				-
				\displaystyle\sum_{j=1}^{n-i}
				a_{i,j}\, x_ix_j
			}{
				1 + \varepsilon\displaystyle\sum_{j=1}^{n}
				c_j (x_j)^2},
		\end{align*}
		whereas when $l=i$, the expression for the $l$th partial derivative is the following
		\begin{equation*}
			\left \{
			\begin{aligned}
				l\leq n-l:&\\
				\frac{\partial}{\partial x_l}Q_{i,\varepsilon}^n\big((x_1,\cdots,x_n)\big) = &\frac{\frac12 \sum\limits_{j=l+1}^n b^i_{l,j-l}a_{l,j-l} x_{j-l}+\frac12 \sum\limits_{j=l+1}^n b^i_{j-l,l}a_{j-l,l} x_{j-l}-a_{l,l}x_l-\sum\limits_{j=1}^{n-l} a_{l,j}x_j}{1+\varepsilon\sum\limits_{j=1}^n c_j (x_j)^2}\\
				+ & \frac{-2\varepsilon c_lx_l}{1+\varepsilon\sum\limits_{j=1}^n c_j (x_j)^2} \times \frac{
					\frac{1}{2}\displaystyle\sum_{j=i+1}^{n}\sum_{k=1}^{j-1}
					b_{j-k,k}^i a_{j-k,k}\,
					x_{j-k} x_k
					-
					\displaystyle\sum_{j=1}^{n-l}
					a_{i,j}\, x_ix_j
				}{
					1 + \varepsilon\displaystyle\sum_{j=1}^{n}
					c_j (x_j)^2}\\
				l> n-l:&\\
				\frac{\partial}{\partial x_l}Q_{i,\varepsilon}^n\big((x_1,\cdots,x_n)\big) = &\frac{\frac12 \sum\limits_{j=l+1}^n b^i_{l,j-l}a_{l,j-l} x_{j-l}+\frac12 \sum\limits_{j=l+1}^n b^i_{j-l,l}a_{j-l,l} x_{j-l}-\sum\limits_{j=1}^{n-l} a_{l,j}x_j}{1+\varepsilon\sum\limits_{j=1}^n c_j (x_j)^2}\\
				+ & \frac{-2\varepsilon c_lx_l}{1+\varepsilon\sum\limits_{j=1}^n c_j (x_j)^2} \times \frac{
					\frac{1}{2}\displaystyle\sum_{j=i+1}^{n}\sum_{k=1}^{j-1}
					b_{j-k,k}^i a_{j-k,k}\,
					x_{j-k} x_k
					-
					\displaystyle\sum_{j=1}^{n-l}
					a_{i,j}\, x_ix_j
				}{
					1 + \varepsilon\displaystyle\sum_{j=1}^{n}
					c_j (x_j)^2}.
			\end{aligned}
			\right .
		\end{equation*}
		Note that the numerators consisting of all linear order terms and the denominators are all second order terms. A simple application of H\"older inequality yields all the first order derivatives are bounded for all $x=(x_1,\cdots,x_n)\in\mathbb{R}^n$. Hence each of the function $Q_{i,\varepsilon}^n$ is globally Lipschitz (Lipschitz constant may depends on $\varepsilon$). Hence classical existence uniqueness theory for parabolic system \cite{amann1995linear, pazy2012semigroups} yields a unique global-in-time solution. The following calculation shows that the function $Q_{i,\varepsilon}^n$ is bounded also.
		\begin{align*}
			\Big|\frac12 \sum\limits_{j=i+1}^n\sum\limits_{k=1}^{j-1} b^i_{j-k,k}a_{j-k,k} x_{j-k}x_{k} -& \sum\limits_{j=1}^{n-i} a_{i,j}x_ix_j\Big| \\
			\leq&  \frac 12\sup\limits_{i,j,k}\{b^i_{j-k,k}\}  \sum\limits_{j=i+1}^n\sum\limits_{k=1}^{j-1} a_{j-k,k}(x_{j-k}^2+x_k^2) + \sum\limits_{j=1}^n a_{i,j} (x_i^2+x_j^2)\\
			\leq & \sup\limits_{i,j,k}\{b^i_{j-k,k}\} \sum\limits_{i=1}^{n}\sum\limits_{j=1}^{n} a_{i,j}x_j^2 +2 \sum\limits_{i=1}^{n}\sum\limits_{j=1}^{n} a_{i,j}x_j^2\\
			\lesssim & \frac{1}{\varepsilon} \left( 1+\varepsilon \sum\limits_{j=1}^n c_j (x_j)^2\right).
		\end{align*}
		Here we write the third line due to symmetricity of the coagulation kernel $a_{i,j}$. Hence we have that:
		\[
		| Q_{i,\varepsilon}^n| \lesssim \frac{1}{\varepsilon}. 
		\]
		Using the $L^{\infty}\to L^{\infty}$ contraction property of the Neumann heat semigroup we deduce the required $L^{\infty}$ bound on the solution. The second order regularity estimate \eqref{eq: SORE} also follows from standered second order regularity estimate of the Neumann heat  \cite{amann1995linear, pazy2012semigroups, quittner2007}.
		\newline
		Note that the source vector $(Q_{1,\varepsilon}^n, \cdots, Q_{n,\varepsilon}^n)$ is quasipositive \eqref{quasipositive}. Hence due to the nonnegativity of the initial data, we have nonnegative unique global-in-time solution. Smoothness follows from Schauder estimates and standard bootstrapping argument for parabolic regularity \cite{quittner2007}. 
		\newline
		Furthermore, the mass conservation estimate follows from the fact that
		\[
		\sum_{i=1}^n iQ_{i,\varepsilon}^n (f_{\varepsilon}^n)=0.
		\]
	\end{proof}

	\section{\textbf{Regularized system}} \label{RS}
	
	In this section, we pass to the limit with respect to the index $n$ in order to obtain
	a global-in-time solution of the regularized approximate system. We note that this regularized
	approximate system involves infinitely many unknowns. More precisely, we have the
	following proposition.
	
	\begin{prop}[Existence and mass conservation for the infinite regularized system]
		\label{prop:infinite_mass}
		For every $\varepsilon>0$, there exists a global-in-time nonnegative strong solution
		\[
		\{f_{i,\varepsilon}\}_{i\ge1}
		\]
		to the system
		\begin{equation}
			\label{eq:RNFE}
			\left\{
			\begin{aligned}
				& \partial_t f_{i,\varepsilon} - d_i \Delta f_{i,\varepsilon}
				= Q_{i,\varepsilon}(f_{\varepsilon})
				&& \text{in } (0,T)\times\Omega \\
				& \nabla f_{i,\varepsilon} \cdot \nu = 0
				&& \text{on } (0,T)\times\partial\Omega \\
				& f_{i,\varepsilon}^n(0,x) = f_i^{\mathrm{in}}(x)
				&& \text{in}\ \Omega,
			\end{aligned}
			\right.
		\end{equation}
		for $i\in\mathbb{N}$, where
		\begin{align}
			\label{eq:RNFEoperator}
			Q_{i,\varepsilon}(f_{\varepsilon})
			= \frac{
				\frac{1}{2}\displaystyle\sum_{j=i+1}^{\infty} \sum_{k=1}^{j-1}
				b_{j-k,k}^i a_{j-k,k}\,
				f_{j-k,\varepsilon} f_{k,\varepsilon}
				-
				\displaystyle\sum_{j=1}^{\infty}
				a_{i,j}\, f_{i,\varepsilon} f_{j,\varepsilon}
			}{
				1 + \varepsilon\displaystyle\sum_{j=1}^{\infty}
				c_j (f_{j,\varepsilon})^2
			}.
		\end{align}
		Moreover,
		\[
		f_{i,\varepsilon} \in C^\infty\big((0,T)\times\Omega\big),
		\qquad
		\| f_{i,\varepsilon} \|_{L^\infty((0,T)\times\Omega)}
		\lesssim \frac{1}{\varepsilon},
		\]
		and
		\[
		\int_\Omega \sum_{i=1}^{\infty} i f_{i,\varepsilon}(t,x)\,dx
		=
		\int_\Omega \sum_{i=1}^{\infty} i f_i^{\mathrm{in}}(x)\,dx,
		\qquad \forall\, t\in[0,T].
		\]
	\end{prop}
	
	\begin{proof}
		We divide the proof into six steps.
		\newline 
		\textbf{Step 1: Approximation by truncated systems.}
		Thanks to Proposition \ref{prop:existence_truncated}, for each $n\ge3$, let $\{f_{i,\varepsilon}^n\}_{i=1}^n$ be the unique global-in-time strong
		nonnegative solution to the truncated regularized system given by \eqref{eq:TRNFE}.
		In particular,
		\begin{equation}
			\label{eq:mass_truncated}
			\int_\Omega \sum_{i=1}^{n} i f_{i,\varepsilon}^n(t,x)\,dx
			=
			\int_\Omega \sum_{i=1}^{n} i f_i^{\mathrm{in}}(x)\,dx,
			\qquad \forall\, t\in[0,T].
		\end{equation}
		\textbf{Step 2: Compactness and convergence.}
		Fix $i\ge1$, $0<\delta<T$, and $1<p<\infty$.
		Thanks to Proposition~\ref{prop:existence_truncated} uniform parabolic regularity estimate implies that
		\[
		f_{i,\varepsilon}^n \in
		W^{1,p}\big((\delta,T);L^p(\Omega)\big)
		\cap
		L^p\big((\delta,T);W^{2,p}(\Omega)\big).
		\]
		By the Aubin--Lions lemma,
		there exists a subsequence (not relabeled) such that
		\[
		f_{i,\varepsilon}^n \to f_{i,\varepsilon}
		\quad \text{strongly in } L^p((\delta,T);W^{1,p}(\Omega))
		\]
		and almost everywhere in $(\delta,T)\times\Omega$ for all $\delta\in(0,T)$.
		
		\medskip
		\textbf{Step 3: Uniform boundedness of the solution and uniform tail summability of the mass.}
		Thanks to Proposition~\ref{prop:existence_truncated}, we have the following the uniform $L^\infty$-bound of the solution to the truncated regularized system \eqref{eq:TRNFE} 
		\[
		\|  f_{i,\varepsilon}^n \|_{L^\infty((0,T)\times\Omega)}
		\lesssim \frac{1}{\varepsilon}.
		\]
		Hence passing to the limit $n\to+\infty$, we conclude that
		\[
		\|  f_{i,\varepsilon} \|_{L^\infty((0,T)\times\Omega)}
		\lesssim \frac{1}{\varepsilon}.
		\]
		Furthermore, summing over the index $i$ in \eqref{eq:TRNFE} and integrating using the summability of $a_{i,j}$ yields
		\[
		\sup_n
		\int_\Omega \sum_{i=1}^{n} i f_{i,\varepsilon}^n(t,x)\,dx
		< \infty,
		\qquad \forall\, t\in[0,T],
		\]
		and
		\begin{align*}
			\int_{\Omega}\sum_{i=M}^{n} i f_{i,\varepsilon}^n(t,x) &\lesssim \Big(\tfrac{1}{\varepsilon}\Big)^2
			\Big(
			\tfrac12\sum_{i=M}^{n}\sum_{j=i+1}^{n}\sum_{k=1}^{j-1}
			i b_{j-k,k}^i a_{j-k,k}
			+
			\sum_{i=M}^{n}\sum_{j=1}^{n-i} i a_{i,j}
			\Big) |\Omega|
			+
			\sum_{i=M}^{n}\|i f_i^{\mathrm{in}}\|_{L^1(\Omega)} 
			\\
			&\lesssim \Big(\tfrac{1}{\varepsilon}\Big)^2
			\Big(
			\tfrac12\sum_{j=M}^{n}\sum_{k=1}^{n}\sum_{i=1}^{j+k-1}
			i b_{j,k}^i a_{j,k}
			+
			\sum_{i=M}^{n}\sum_{j=1}^{n-i} i a_{i,j}
			\Big) |\Omega|
			+
			\sum_{i=M}^{n}\|i f_i^{\mathrm{in}}\|_{L^1(\Omega)} 
			\\
			& \lesssim \Big(\tfrac{1}{\varepsilon}\Big)^2
			\Big(
			\tfrac12\sum_{j=M}^{n}\sum_{k=1}^{n}(j+k) a_{j,k}
			+
			\sum_{i=M}^{n}\sum_{j=1}^{n-i} i a_{i,j}
			\Big) |\Omega|
			+
			\sum_{i=M}^{n}\|i f_i^{\mathrm{in}}\|_{L^1(\Omega)} \\
			& \lesssim \delta_M \to 0 \ \mbox{as}\ M\to\infty.
		\end{align*}

		\medskip
		\textbf{Step 4: Passage to the limit in the mass.}
		Fix $t\in[0,T]$ and decompose
		\[
		\sum_{i=1}^{n} i f_{i,\varepsilon}^n(t,x)
		=
		\sum_{i=1}^{M} i f_{i,\varepsilon}^n(t,x)
		+
		\sum_{i=M+1}^{n} i f_{i,\varepsilon}^n(t,x).
		\]
		Uniform tail summability from the previous step yields
		\[
		\int_\Omega \sum_{i=M+1}^{n} i f_{i,\varepsilon}^n(t,x)\,dx
		\le \delta_M, \ \forall \, n>M.
		\]
		For fixed $M$, almost everywhere convergence implies
		\[
		\sum_{i=1}^{M} i f_{i,\varepsilon}^n(t,x)
		\to
		\sum_{i=1}^{M} i f_{i,\varepsilon}(t,x),
		\]
		and by dominated convergence,
		\[
		\lim_{n\to\infty}
		\int_\Omega \sum_{i=1}^{M} i f_{i,\varepsilon}^n(t,x)\,dx
		=
		\int_\Omega \sum_{i=1}^{M} i f_{i,\varepsilon}(t,x)\,dx.
		\]
		Letting first $n\to\infty$ and then $M\to\infty$, we conclude that
		\[
		\int_\Omega \sum_{i=1}^{\infty} i f_{i}^{\rm{in}}(x)\,dx= \lim_{n\to\infty}
		\int_\Omega \sum_{i=1}^{n} i f_{i,\varepsilon}^n(t,x)\,dx
		=
		\int_\Omega \sum_{i=1}^{\infty} i f_{i,\varepsilon}(t,x)\,dx.
		\]
		
		\textbf{Step 5: Passing the limit in the weak formulation of truncated regularized system.} Consider the weak formulation corresponding to \eqref{eq:TRNFE}: for every $\psi\in C_c^{\infty}([0,T)\times\Omega)$,
		\[
		\int_\Omega
		\psi(0,x)\, f_i^{\mathrm{in}}(x)\, dx
		+
		\int_0^T \int_\Omega
		\left(-\partial_t \psi\right) f_{i,\varepsilon}^n
		\,dx\,dt
		+
		\int_0^T \int_\Omega
		\nabla \psi \cdot \nabla f_{i,\varepsilon}^n
		\,dx\,dt=
		\int_0^T \int_\Omega
		\psi \, Q_{i,\varepsilon}^n(f_{\varepsilon}^n)
		\,dx\,dt.
		\]
		
		\textbf{Substep 1: Passing to the limit in the right-hand side of the weak formulation.}
		The solution $f_{i,\varepsilon}^n$ and it's gradient can be expressed as:
		\begin{align*}
			f_{i,\varepsilon}^n= \int_{\Omega}G_i(t,0,x,y)f_i^{\rm{in}}(y) \, dy+ \int_0^t \int_{\Omega} G_i(t,s,x,y) Q_{i,\varepsilon}^n(f_{\varepsilon}^n)(s,y) \, dy\, ds,\\
			\nabla f_{i,\varepsilon}^n= \int_{\Omega}\nabla G_i(t,0,x,y)f_i^{\rm{in}}(y) \, dy+ \int_0^t \int_{\Omega}\nabla  G_i(t,s,x,y) Q_{i,\varepsilon}^n(f_{\varepsilon}^n)(s,y) \, dy\, ds,
		\end{align*}
		where $G_i(t,s,x,y)$ be the Neumann Green function corresponding to the heat operator $\partial_t-d_i\Delta$. We use the following  Green's function estimate as described in \cite{mora1983}. There exists a constant $\tilde{\kappa}>0$, depending only on the domain, such that
		\begin{align}\label{Heat kernel estimate}
			0\leq \left| D_x^m G_{d_i}(t,s,x,y)\right| \leq \frac{\tilde{\kappa}}{(4d_i\pi( t-s))^\frac{N+m}{2}}e^{-\kappa \frac{\Vert x-y \Vert^2}{d_i(t-s)}}\leq g_{m,d_i}(t-s,x-y) \qquad 0\leq s<t.
		\end{align}
		Here, the constants $\kappa>0$, depending only on $\Omega$ and independent of the diffusion coefficient. Furthermore, the function $g_{m,d_i}(t-s,x-y)$ is defined in the following way:
		\begin{align}\label{Heat kernel bound function}
			g_{m,d_i}(t,x):= \frac{\tilde{\kappa}}{(4d_i\pi\vert t \vert)^\frac{N+m}{2}}e^{-\kappa \frac{\Vert x\Vert^2}{d_i\vert t\vert}}, \qquad (t,x)\in \mathbb{R}\times\mathbb{R}^{N}.
		\end{align}
		Note that
		\begin{align}\label{Heat kernel integral estimate}
			\left\Vert g_{m,d_i}\right\Vert_{\mathrm{L}^{z}((-T,T)\times\mathbb{R}^N)} \lesssim d_i^{-\frac{(N+m)z-N}{2z}}, \qquad \forall z\in\left[1,1+\frac{1}{N+1}\right), \ \mbox{for} \ m=0,1,
		\end{align}
		and 
		\begin{align}\label{Heat kernel integral estimate, space only}
			\left\Vert g_{m,d_i}\right\Vert_{\mathrm{L}^{z}(\mathbb{R}^N)} \lesssim d_i^{-\frac{(N+m)z-N}{2z}}t^{-\frac{(N+m)z-N}{2z}}, \qquad \forall z\in\left[1,1+\frac{1}{N+1}\right), \ \mbox{for} \ m=0,1.
		\end{align}
		We move onto find an integral estimate of the solution to the truncated regularized system \eqref{eq:TRNFE}. Using triangle inequality we have that:
		\begin{align*}
			\| f_{i,\varepsilon}^n\|_{L^{z}((0,T); W^{1,z}((\Omega))} \leq &\left \| \int_{\Omega}G_i(t,0,x,y)f_i^{\rm{in}}(y) \, dy \right\|_{L^{z}((0,T); W^{1,z}((\Omega))}\\
			&+ \left \| \int_0^t \int_{\Omega}  G_i(t,s,x,y) Q_{i,\varepsilon}^n(f_{\varepsilon}^n)(s,y) \, dy\, ds \right\|_{L^{z}((0,T); W^{1,z}((\Omega))}. 
		\end{align*}
		Let us denote $\displaystyle{\mathcal{E}:= \int_0^t \int_{\Omega}  G_i(t,s,x,y) Q_{i,\varepsilon}^n(f_{\varepsilon}^n)(s,y) \, dy\, ds }$. Let us further denote that
		\begin{equation*}
			\tilde{Q}_{i,\varepsilon}^n (s,y) =
			\left\{
			\begin{aligned}
				&|Q_{i,\varepsilon}^n (f_{\varepsilon}^n)| \qquad (s,y)\in (0,T)\times\Omega\\
				&0 \qquad \qquad \qquad \mbox{outside}.
			\end{aligned}
			\right .
		\end{equation*}
		Then $\displaystyle{\|\mathcal{E}\|_{L^{z}((0,T); W^{1,z}((\Omega))}\leq \| \tilde{g}_{0,d_i}+\tilde{g}_{1,d_i}\|_{L^{1}(\mathbb{R}\times\mathbb{R}^N)}\| \tilde{Q}_{i,\varepsilon}^n\|_{L^{z}(\mathbb{R}\times\mathbb{R^N})}}$, where $\tilde{g}_{m,d_i}$ is the extension of $g_{m,d_i}$, defined in the following way:
		\begin{equation*}
			\tilde{g}_{m,d_i} (s,y) =
			\left\{
			\begin{aligned}
				&g_{m,d_i} (f_{\varepsilon}^n) \qquad (s,y)\in (-T,T)\times\Omega\\
				&0 \qquad \qquad \qquad \mbox{outside}.
			\end{aligned}
			\right .
		\end{equation*}
		Hence
		\[
		\|\mathcal{E}\|_{L^{z}((0,T); W^{1,z}((\Omega))}\lesssim \frac{d_i^{-\frac12}}{\varepsilon}. 
		\]
		Similarly we can show the first expression can also be dominated by $\displaystyle{\frac{d_i^{-\frac12}}{\varepsilon}}$. Hence
		\[
		\| f_{i,\varepsilon}^n\|_{L^{z}((0,T); W^{1,z}((\Omega))} \lesssim \frac{d_i^{-\frac12}}{\varepsilon}. 
		\]
		Fix $z\in\left(1,1+\frac{1}{N+1}\right)$. Using weak lower semicontinuty of the $L^p$ norm, we have that
		\[
		\| f_{i,\varepsilon}\|_{L^{z}((\delta,T); W^{1,z}((\Omega))} \leq \liminf\limits_{n\in\mathbb{N}}\| f_{i,\varepsilon}^n\|_{L^{z}((\delta,T); W^{1,z}((\Omega))}\lesssim \frac{d_i^{-\frac12}}{\varepsilon}.
		\]
		Taking $\delta\to 0$, yields
		\[
		\| f_{i,\varepsilon}\|_{L^{z}((0,T); W^{1,z}((\Omega))} \lesssim \frac{d_i^{-\frac12}}{\varepsilon}.
		\]
		Next, we concentrate on passing the limit in each individual expression present in the L.H.S of the weak formulation. Note that the first expression is independent of the index `$n$'. So we will consider the next two expressions only. We analyze the second expression as follows:
		\begin{align*}
			\Big|\int_0^T \int_\Omega
			\left(-\partial_t \psi\right) f_{i,\varepsilon}^n
			\,dx\,dt - & \int_0^T \int_\Omega
			\left(-\partial_t \psi\right) f_{i,\varepsilon}
			\,dx\,dt\Big| \\
			\leq \Big| \int_{\delta}^T \int_\Omega
			\left(-\partial_t \psi\right) f_{i,\varepsilon}^n
			& \,dx\,dt - \int_{\delta}^T \int_\Omega
			\left(-\partial_t \psi\right) f_{i,\varepsilon}
			\,dx\,dt\Big| \\
			+\sup |\partial_t \psi| & \left(\| f_{i,\varepsilon}^n\|_{L^1((0,\delta)\times\Omega)}+\| f_{i,\varepsilon}\|_{L^1((0,\delta)\times\Omega)}\right)\to \, 0 \ \text{as}\ \delta\to 0.
		\end{align*}
		Similarly, we analyze the third expression. It yields:
		\[
		\lim_{n\to\infty}\int_0^T \int_\Omega
		\nabla \psi \cdot \nabla f_{i,\varepsilon}^n
		\,dx\,dt= \int_0^T \int_\Omega
		\nabla \psi \cdot \nabla f_{i,\varepsilon}
		\,dx\,dt.
		\]
		Hence, for all $\psi\in C^{\infty}_c([0,T)\times\Omega)$, passing the limit with respect to the index `$n$' in the L.H.S, we conclude that:
		\begin{align*}
			\lim_{n\to\infty}&\left( 
			\int_\Omega
			\psi(0,x)\, f_i^{\mathrm{in}}(x)\, dx
			+
			\int_0^T \int_\Omega
			\left(-\partial_t \psi\right) f_{i,\varepsilon}^n
			\,dx\,dt
			+
			\int_0^T \int_\Omega
			\nabla \psi \cdot \nabla f_{i,\varepsilon}^n
			\,dx\,dt
			\right)\\
			=& \int_\Omega
			\psi(0,x)\, f_i^{\mathrm{in}}(x)\, dx
			+
			\int_0^T \int_\Omega
			\left(-\partial_t \psi\right) f_{i,\varepsilon}
			\,dx\,dt
			+
			\int_0^T \int_\Omega
			\nabla \psi \cdot \nabla f_{i,\varepsilon}
			\,dx\,dt.
		\end{align*}
		\medskip

		\textbf{Substep 2: Passing to the limit in the right-hand side of the weak formulation.} We write
		\[
		Q_{i,\varepsilon}^n(f_{\varepsilon}^n)
		=
		\frac{F_{i,\varepsilon}^n(f_{\varepsilon}^n)-L_{i,\varepsilon}^n(f_{\varepsilon}^n)}
		{1+\varepsilon\displaystyle\sum_{j=1}^{n} c_j (f_{j,\varepsilon}^n)^2},
		\]
		where
		\[
		F_{i,\varepsilon}^n
		=
		\frac12\sum_{j=i+1}^{n}\sum_{k=1}^{j-1}
		b_{j-k,k}^i a_{j-k,k}
		f_{j-k,\varepsilon}^n f_{k,\varepsilon}^n,
		\qquad
		L_{i,\varepsilon}^n
		=
		\sum_{j=1}^{n-i} a_{i,j} f_{i,\varepsilon}^n f_{j,\varepsilon}^n.
		\]
		We focus on the gain term \(F_{i,\varepsilon}^n\).
		Fix \(M>i\) and decompose
		\[
		F_{i,\varepsilon}^n
		=
		\frac12\sum_{j=i+1}^{2M}\sum_{k=1}^{j-1}
		b_{j-k,k}^i a_{j-k,k}
		f_{j-k,\varepsilon}^n f_{k,\varepsilon}^n
		+
		\frac12\sum_{j=2M+1}^{n}\sum_{k=1}^{j-1}
		b_{j-k,k}^i a_{j-k,k}
		f_{j-k,\varepsilon}^n f_{k,\varepsilon}^n.
		\]
		Using the uniform bound
		\[
		\|f_{j,\varepsilon}^n\|_{L^\infty((0,T)\times\Omega)}
		\lesssim \frac{1}{\varepsilon},
		\]
		and the boundedness \(b_{j,k}^i\le B\) (assumption~\eqref{strong assumption}), we estimate the tail as
		\[
		0 \le
		\sum_{j=2M+1}^{n}\sum_{k=1}^{j-1}
		b_{j-k,k}^i a_{j-k,k}
		f_{j-k,\varepsilon}^n f_{k,\varepsilon}^n
		\lesssim
		\frac{1}{\varepsilon^2}
		\sum_{j=2M+1}^{n}\sum_{k=1}^{j-1} a_{j-k,k}.
		\]
		By the summability assumption on \(a_{j,k}\),
		\[
		\sum_{j=2M+1}^{n}\sum_{k=1}^{j-1} a_{j-k,k}
		\le
		2\sum_{j=M}^{\infty}\sum_{k=1}^{\infty} a_{j,k}
		=: \delta_M,
		\qquad \delta_M \to 0 \text{ as } M\to\infty.
		\]
		Hence,
		\[
		0 \le F_{i,\varepsilon}^n
		\le
		\frac12\sum_{j=i+1}^{2M}\sum_{k=1}^{j-1}
		b_{j-k,k}^i a_{j-k,k}
		f_{j-k,\varepsilon}^n f_{k,\varepsilon}^n
		+ \delta_M.
		\]
		Taking the upper limit with respect to the index `$n$' and using the strong convergence
		\(f_{j,\varepsilon}^n \to f_{j,\varepsilon}\) for each fixed index, we obtain
		\[
		\limsup_{n\to\infty} F_{i,\varepsilon}^n
		\le
		\frac12\sum_{j=i+1}^{2M}\sum_{k=1}^{j-1}
		b_{j-k,k}^i a_{j-k,k}
		f_{j-k,\varepsilon} f_{k,\varepsilon}
		+ \delta_M.
		\]
		Letting \(M\to\infty\), we conclude
		\[
		\limsup_{n\to\infty} F_{i,\varepsilon}^n
		\le
		\frac12\sum_{j=i+1}^{\infty}\sum_{k=1}^{j-1}
		b_{j-k,k}^i a_{j-k,k}
		f_{j-k,\varepsilon} f_{k,\varepsilon}
		=: F_{i,\varepsilon}.
		\]
		On the other hand, since
		\[
		F_{i,\varepsilon}^n
		\ge
		\frac12\sum_{j=i+1}^{2M}\sum_{k=1}^{j-1}
		b_{j-k,k}^i a_{j-k,k}
		f_{j-k,\varepsilon}^n f_{k,\varepsilon}^n,
		\]
		we infer
		\[
		\liminf_{n\to\infty} F_{i,\varepsilon}^n
		\ge
		\frac12\sum_{j=i+1}^{2M}\sum_{k=1}^{j-1}
		b_{j-k,k}^i a_{j-k,k}
		f_{j-k,\varepsilon} f_{k,\varepsilon}.
		\]
		Letting \(M\to\infty\) again yields
		\[
		\liminf_{n\to\infty} F_{i,\varepsilon}^n
		\ge F_{i,\varepsilon}.
		\]
		Combining the limsup and liminf estimates, we conclude that
		\[
		\lim_{n\to\infty} F_{i,\varepsilon}^n = F_{i,\varepsilon}.
		\]
		Similarly, it can be shown that the loss term
		\[
		L_{i,\varepsilon}^n(f_\varepsilon^n)
		=
		\sum_{j=1}^{n-i} a_{i,j}\, f_{i,\varepsilon}^n f_{j,\varepsilon}^n
		\]
		converges almost everywhere in $(0,T)\times\Omega$ to
		\[
		L_{i,\varepsilon}(f_\varepsilon)
		=
		\sum_{j=1}^{\infty} a_{i,j}\, f_{i,\varepsilon} f_{j,\varepsilon}.
		\]
		Indeed, for each fixed $i$ the summation is finite and the strong
		convergence $f_{j,\varepsilon}^n \to f_{j,\varepsilon}$ yields the
		result. Moreover, since $f_{j,\varepsilon}^n$ is uniformly bounded (Proposition \ref{prop:existence_truncated}) and converges to $f_{j,\varepsilon}$ almost everywhere, the denominator converges almost everywhere to
		\[
		1+\varepsilon\sum_{j=1}^{\infty} c_j (f_{j,\varepsilon})^2.
		\]
		Consequently,
		\[
		Q_{i,\varepsilon}^n(f_\varepsilon^n)
		\longrightarrow
		Q_{i,\varepsilon}(f_\varepsilon)
		:=
		\frac{F_{i,\varepsilon}(f_\varepsilon)-L_{i,\varepsilon}(f_\varepsilon)}
		{1+\varepsilon\displaystyle\sum_{j=1}^{\infty} c_j (f_{j,\varepsilon})^2}
		\quad \text{a.e. in } (0,T)\times\Omega.
		\]
		Note that the sequence $\{Q_{i,\varepsilon}^n(f_\varepsilon^n)\}_n$
		is uniformly bounded in $L^{\infty}((0,T)\times\Omega)$.
		Hence, by the dominated convergence theorem,
		\[
		Q_{i,\varepsilon}^n(f_\varepsilon^n)
		\longrightarrow
		Q_{i,\varepsilon}(f_\varepsilon)
		\quad \text{in } L^1((0,T)\times\Omega).
		\]
		Therefore, for all test functions
		$\psi \in C_c^{\infty}([0,T)\times\Omega)$,
		\[
		\int_0^T \int_\Omega
		\psi\, Q_{i,\varepsilon}^n(f_\varepsilon^n)\,dx\,dt
		\longrightarrow
		\int_0^T \int_\Omega
		\psi\, Q_{i,\varepsilon}(f_\varepsilon)\,dx\,dt.
		\]
		\textbf{Step 6: Regularity.}
		Combining the above limits, we obtain the existence of a global-in-time nonnegative weak
		solution to the regularized system \eqref{eq:RNFE}. We note that the solutions
		$f_{i,\varepsilon}$ are uniformly bounded for each $\varepsilon$ and for all
		$i \in \mathbb{N}$. Moreover, the estimate \eqref{eq:mass_truncated}, together with the
		assumption on the initial mass, yields conservation of mass for solutions of the
		regularized system \eqref{eq:RNFE}. Standard parabolic bootstrapping arguments then imply
		that
		\[
		f_{i,\varepsilon} \in C^\infty((0,T)\times\Omega).
		\]
		Hence, the solution is classical.
	\end{proof}

	\section{\textbf{Uniform estimates and compactness of the regularized solutions}}\label{UECRS}
	
	In this section, we analyze the compactness properties of the solutions to the regularized
	system \eqref{eq:RNFE}. We will make use of the $L^1$ compactness result as described
	in \cite{Pierre2010}. In order to apply this result, it is necessary that the source terms
	belong to $L^1((0,T)\times\Omega)$ uniformly. We begin with the following proposition.
	\begin{prop}
		\label{prop:dual_estimate}
		Assume that
		\[
		\sum_{i=1}^{\infty} i f_i^{\mathrm{in}} \in L^2(\Omega),
		\qquad
		\sup_{i\in\mathbb{N}^*} d_i < +\infty.
		\]
		Then, for all $T>0$, the solutions to the regularized system \eqref{eq:RNFE} satisfy the following estimate
		\begin{equation}
			\label{eq:dual_estimate}
			\int_0^T \int_\Omega
			\left( \sum_{i=1}^{\infty} i d_i f_{i,\varepsilon}(t,x) \right)
			\left( \sum_{i=1}^{\infty} i f_{i,\varepsilon}(t,x) \right)
			\,dx\,dt
			\;\lesssim\;
			\left( \sup_{i\ge1} d_i \right)
			\left\|
			\sum_{i=1}^{\infty} i f_i^{\mathrm{in}}
			\right\|_{L^2(\Omega)}^{2}.
		\end{equation}
	\end{prop}
	The proof of the proposition follows along the same lines as the $L^2$ estimates
	established in\cite{desvillettes2007global, canizo2010absence}. Since the arguments are
	essentially the same, we omit the details. The above proposition helps us to obtain the required  uniform $L^1((0,T)\times\Omega)$ estimate on the source terms. 
	\begin{lem}[Uniform $L^1$--bound with respect to $\varepsilon$] \label{bound on NR operator}
		Let $\varepsilon\in(0,1)$ and let $f_{i,\varepsilon}$ be the solution of the regularized system \eqref{eq:RNFE}. Then for every $T>0$ and for all $\varepsilon>0$, we have the following estimate
		\[
		\sum_{i=1}^\infty
		\left\|
		\frac{Q_{i,\varepsilon}(f_{\varepsilon})}{d_i}
		\right\|_{L^1((0,T)\times\Omega)}
		\le C,
		\]
		where the constant $C>0$ is independent of $\varepsilon$.
		In particular, for every $i\ge1$ and for all $\varepsilon>0$,
		\[
		Q_{i,\varepsilon}(f_{\varepsilon})\in L^1((0,T)\times\Omega) \ \text{uniformly}.
		\]
	\end{lem}
	
	\begin{proof}
		By definition \eqref{eq:RNFEoperator}, the denominator satisfies
		\[
		1+\varepsilon\sum_{j=1}^{\infty} c_j (f_{j,\varepsilon})^2 \ge 1
		\quad \text{for all } \varepsilon>0.
		\]
		Moreover,
		\begin{align*}
			\sum_{i=1}^\infty \left\|
			\frac{Q_{i,\varepsilon}(f_{\varepsilon})}{d_i}
			\right\|_{L^1((0,T)\times\Omega)}
			&\le
			\frac{1}{2}\sum_{i=1}^\infty\left\|
			\sum_{j=i+1}^{\infty} \sum_{k=1}^{j-1}
			\frac{b_{j-k,k}^i a_{j-k,k}}{k(j-k)d_{k}d_i}\,
			(j-k)f_{j-k,\varepsilon} kd_kf_{k,\varepsilon}
			\right\|_{L^1((0,T)\times\Omega)}
			\\
			&\quad+\sum_{i=1}^\infty
			\left\|
			\sum_{j=1}^{\infty}
			\frac{a_{i,j}}{ijd_jd_i}\, if_{i,\varepsilon} jd_jf_{j,\varepsilon}
			\right\|_{L^1((0,T)\times\Omega)}.
		\end{align*}
		By Proposition~\ref{prop:dual_estimate}, there exists a constant $C>0$, independent of $i$ and $\varepsilon$, such that
		\begin{align*}
			\sum_{i=1}^\infty \left\|
			\frac{Q_{i,\varepsilon}(f_{\varepsilon})}{d_i}
			\right\|_{L^1((0,T)\times\Omega)}
			&\le
			C\sum_{i=1}^\infty
			\sum_{j=i+1}^{\infty} \sum_{k=1}^{j-1}
			\frac{b_{j-k,k}^i a_{j-k,k}}{k(j-k)d_{k}d_i}+C\sum_{i=1}^\infty
			\sum_{j=1}^{\infty}
			\frac{a_{i,j}}{ijd_jd_i}\\
			&\le C
			\sum_{j=1}^{\infty} \sum_{k=1}^{\infty}\sum_{i=1}^{j+k-1}
			\frac{b_{j,k}^i a_{j,k}}{kjd_{k}d_i}+C\sum_{i=1}^\infty
			\sum_{j=1}^{\infty}
			\frac{a_{i,j}}{ijd_jd_i} <+\infty.
		\end{align*}
	\end{proof}
	Next, we recall a compactness result from\cite{baras1984, bothe2010, Pierre2010}, which we
	use to establish the convergence of the sequence of solutions to the regularized system \eqref{eq:RNFE}.
	\begin{lem}\label{lem:compactness}
		Let $d>0$. The mapping $( \mathcal{F}^{\rm{in}}, \mathcal{Q})\mapsto \mathcal{F}$, where $\mathcal{F}$ is the solution of
		\begin{equation}\label{eq:compactness_heat}
			\begin{cases}
				\partial_t \mathcal{F} - d \Delta \mathcal{F} = \mathcal{Q} & \text{in } (0,T)\times \Omega \\
				\nabla \mathcal{F}\cdot\nu = 0 & \text{on }  (0,T) \times \partial\Omega \\
				\mathcal{F}(0,\cdot) =  \mathcal{F}^{\rm{in}} & \text{in } \Omega,
			\end{cases}
		\end{equation}
		is compact from $L^{1}(\Omega)\times L^{1}((0,T)\times \Omega)$ into
		$L^{1}((0,T)\times \Omega)$, and even into
		$L^{1}\big((0,T); W^{1,1}(\Omega)\big)$.
	\end{lem}
	We are now well equipped to establish compactness for the solutions of the
	regularized system \eqref{eq:RNFE}. More precisely, we have the following lemma.
	\begin{lem}[Compactness]
		\label{lem:compactness of sequence}
		Let $f_{i,\varepsilon}$ be the solution of the regularized system \eqref{eq:RNFE} for all $i\in\mathbb{N}$ and $\varepsilon>0$. Assume that, for each $i\ge1$, the family of source terms
		$\{Q_{i,\varepsilon}(f_{\varepsilon})\}_{\varepsilon>0}$ is uniformly bounded in
		$L^1((0,T)\times\Omega)$ with respect to $\varepsilon$.
		Then, for each fixed $i\ge1$, the sequence
		$\{f_{i,\varepsilon}\}_{\varepsilon>0}$ is relatively compact in
		$L^1((0,T)\times\Omega)$ and, up to extraction of a subsequence,
		\[
		f_{i,\varepsilon} \to f_i
		\quad \text{in } L^1((0,T); W^{1,1}(\Omega)).
		\]
	\end{lem}
	
	\begin{proof}
		For each fixed $i\ge1$, the function $f_{i,\varepsilon}$ solves the linear
		parabolic problem
		\[
		\partial_t f_{i,\varepsilon} - d_i \Delta f_{i,\varepsilon}
		= Q_{i,\varepsilon}(f_{\varepsilon})
		\quad \text{in } (0,T)\times\Omega,
		\]
		with homogeneous Neumann boundary condition and initial data $f_i^{\mathrm{in}}\in L^1(\Omega)$.
		By Lemma~\ref{bound on NR operator}, the right-hand side $Q_{i,\varepsilon}(f_{\varepsilon})$ is uniformly
		bounded in $L^1((0,T)\times\Omega)$ with respect to $\varepsilon>0$.
		Therefore, by Lemma~\ref{lem:compactness}
		\[
		(f_i^{\mathrm{in}}, Q_{i,\varepsilon}(f_{\varepsilon})) \mapsto f_{i,\varepsilon}
		\]
		is compact from $L^1(\Omega)\times L^1((0,T)\times\Omega)$ to
		$L^1((0,T)\times\Omega)$, and even to
		$L^1((0,T); W^{1,1}(\Omega))$.
		Consequently, the family $\{f_{i,\varepsilon}\}_{\varepsilon>0}$ is
		relatively compact in $L^1((0,T)\times\Omega)$, and the stated convergence
		follows up to a subsequence.
	\end{proof}
	
	\section{\textbf{Control of the Regularized Term via a Truncation Procedure}} \label{TPCRT}
	
	At this stage, for each $\varepsilon>0$, the regularized system admits a global-in-time strong solution $(f_{i,\varepsilon})_{i\ge1}$, and for each fixed $i\ge1$ the
	family $\{f_{i,\varepsilon}\}_{\varepsilon>0}$ is relatively compact in
	$L^1((0,T)\times\Omega)$. However, this compactness is not sufficient to pass to
	the limit $\varepsilon\to0$ in the system \eqref{eq:RNFE}, nor even in it's weak formulation. Indeed, the right-hand side of the regularized system \eqref{eq:RNFE} contains the regularized term
	\[
	\varepsilon \sum_{j=1}^{\infty} c_j (f_{j,\varepsilon})^2.
	\]
	Although $\sum_{j=1}^{\infty} c_j < \infty$, no uniform bound is available for
	$\sum_{j=1}^{\infty} c_j (f_{j,\varepsilon})^2$ as $\varepsilon\to0$.
	\newline
	To overcome this difficulty, we employ an additional truncation argument following the
	approach of Pierre~\cite{Pierre2010}. This allows us to control the nonlinear terms
	pointwise and to pass to the limit in a suitable weak sense. Here, the truncation is
	performed on the range of the solutions. More precisely, we consider level sets of the
	form
	\[
	\left\{\, f_{i,\varepsilon} + \eta \sum_{\substack{j=1 \\ j \neq i}}^{\infty} f_j \le
	\mathcal{K} \,\right\},
	\]
	where $\mathcal{K} > 0$ is arbitrary. The role of the parameter $\eta > 0$ is to recover
	the function $f_i$ in the limit $\eta \to 0$, the
	influence of the other components in the level sets vanishes. Another difficulty arises from differentiating the infinite sum with respect to both time and space variables, as it is not well behaved near the initial time $t = 0$. To overcome this issue, we introduce a time-shift parameter $\delta > 0$ and recover the initial term
	in the limit as $\delta \to 0$. We begin by defining the following truncation. The level sets of the solution are then recovered from the support of its derivative.

	\subsection{Definition of the truncation}
	
	For $m>0$, we introduce a truncation function $T_m:\mathbb{R}\to\mathbb{R}$ such
	that
	\begin{align}\label{m_trunction}
		T_m(\sigma)=\sigma \quad \text{for } 0\le \sigma \le m-1,
		\qquad
		T_m(\sigma)=m \quad \text{for } \sigma \ge m,
	\end{align}
	and which is extended smoothly so that
	\[
	T_m \in C^2(\mathbb{R}), \qquad
	0 \le T_m'(\sigma) \le 1,
	\qquad
	-1 \le T_m''(\sigma) \le 0
	\quad \text{for } \sigma \ge 0.
	\]
	For each $i\ge1$ and $\varepsilon>0$, we define
	\[
	\omega_{i,\varepsilon}
	:= f_{i,\varepsilon}
	+ \eta \sum_{j\neq i} f_{j,\varepsilon},
	\qquad
	v^m_{i,\varepsilon}
	:= T_m(\omega_{i,\varepsilon}),
	\]
	where $\eta>0$ is fixed.
	\newline
	By construction, the truncated variables satisfy the uniform bound
	\[
	0 \le v^m_{i,\varepsilon} \le m
	\quad \text{a.e. in } (0,T)\times\Omega,
	\]
	independently of $\varepsilon$.
	
	\subsection{Justification of weak time and spatial derivatives}
	
	Before deriving the differential inequality satisfied by the truncated
	quantities, we rigorously justify that the weak time derivative and the weak
	spatial derivatives of the relevant infinite sums are well defined.
	
	\subsubsection*{Regularity and uniform estimates:}
	
	Thanks to Proposition~\ref{prop:infinite_mass} and the second order regularity estimate \cite{quittner2007}, for every $\varepsilon>0$,
	$i\ge1$, $0<\delta<T$, and $1<p<\infty$, the strong solutions satisfy
	\begin{align}
		\label{eq:Wp_est}
		\| f_{i,\varepsilon} \|_{L^p\!\big((0,T); W^{1,p}(\Omega)\big)}
		&\le 
		\frac{C}{\varepsilon^{2}}
		\left(
		\sum_{j=i+1}^{\infty}\sum_{k=1}^{j-1}
		b_{j-k,k}^{\,i}\, a_{j-k,k}
		+
		\sum_{j=1}^{\infty} a_{i,j}
		\right)+C\,\| f_i^{\mathrm{in}} \|_{L^{p}(\Omega)},
		\\[1ex]
		\label{eq:dt_est}
		\| \partial_t f_{i,\varepsilon} \|_{L^{p}\!\big((\delta,T)\times\Omega\big)}
		&\le
		\frac{C}{d_i\,\varepsilon^{2}}
		\left(
		\sum_{j=i+1}^{\infty}\sum_{k=1}^{j-1}
		b_{j-k,k}^{\,i}\, a_{j-k,k}
		+
		\sum_{j=1}^{\infty} a_{i,j}
		\right)
		+ C\,\| f_i^{\mathrm{in}} \|_{L^{p}(\Omega)} .
	\end{align}
	The constant $C>0$ is independent of $i$ and $\varepsilon$. The summability assumptions on the coefficients (see assumption~\ref{strong assumption}) imply that the right-hand sides
	of \eqref{eq:Wp_est}--\eqref{eq:dt_est} are summable with respect to $i$.

	\subsubsection*{Time derivative}
	
	Fix $0<\delta<T$. Letting
	$n\to\infty$, we obtain
	\[
	\sum_{i=1}^{n} \partial_t f_{i,\varepsilon}
	\;\longrightarrow\;
	\sum_{i=1}^{\infty} \partial_t f_{i,\varepsilon}
	\quad \text{in } L^2\!\big((\delta,T)\times\Omega\big).
	\]
	Hence the series $\sum_{i=1}^{\infty} \partial_t f_{i,\varepsilon}$ converges in
	$L^1((\delta,T)\times\Omega)$. Let $\phi\in C_c^\infty((\delta,T)\times\Omega)$. By dominated convergence, we
	have
	\begin{align*}
		\int_{\delta}^{T}\!\!\int_{\Omega}
		\Big(\sum_{i=1}^{\infty} f_{i,\varepsilon}\Big)\,\partial_t \phi
		&=
		\lim_{n\to\infty}
		\int_{\delta}^{T}\!\!\int_{\Omega}
		\Big(\sum_{i=1}^{n} f_{i,\varepsilon}\Big)\,\partial_t \phi
		\\
		&=
		-\lim_{n\to\infty}
		\int_{\delta}^{T}\!\!\int_{\Omega}
		\Big(\sum_{i=1}^{n} \partial_t f_{i,\varepsilon}\Big)\,\phi
		\\
		&=
		-\int_{\delta}^{T}\!\!\int_{\Omega}
		\Big(\sum_{i=1}^{\infty} \partial_t f_{i,\varepsilon}\Big)\,\phi .
	\end{align*}
	Therefore,
	\[
	\partial_t\!\left(\sum_{i=1}^{\infty} f_{i,\varepsilon}\right)
	=
	\sum_{i=1}^{\infty} \partial_t f_{i,\varepsilon}.
	\]
	
	\subsubsection*{Spatial gradient}
	
	Similarly using second order regularity estimate \cite{quittner2007} and summability assumption on the coefficients (see assumption~\eqref{strong assumption}) helps us conclude 
	\begin{align*}
		\sum_{i=1}^{n} \frac{\delta}{\delta x_l} f_{i,\varepsilon}
		\;\longrightarrow\;
		\sum_{i=1}^{\infty} \frac{\delta}{\delta x_l} f_{i,\varepsilon}
		\quad &\text{in } L^2\!\big((\delta,T)\times\Omega\big), \, \forall\, l=1,\cdots,N,\\
		\sum_{i=1}^{n} \frac{\delta^2}{\delta x_r\delta x_l} f_{i,\varepsilon}
		\;\longrightarrow\;
		\sum_{i=1}^{\infty} \frac{\delta^2}{\delta x_r\delta x_l} f_{i,\varepsilon}
		\quad &\text{in } L^2\!\big((\delta,T)\times\Omega\big), \, \forall\, l,r=1,\cdots,N.
	\end{align*}
	Hence $\sum_{i=1}^{\infty} \frac{\delta}{\delta x_l} f_{i,\varepsilon}$ is well defined in
	$L^1((0,T)\times\Omega)$. For any test function
	$\boldsymbol{\psi}\in C_c^\infty((0,T)\times\Omega;\mathbb{R}^N)$, we compute
	\begin{align*}
		\int_0^T\!\!\int_\Omega
		\Big(\sum_{i=1}^{\infty} f_{i,\varepsilon}\Big)
		\frac{\delta}{\delta x_l}\!\boldsymbol{\psi}
		&=
		\lim_{n\to\infty}
		\int_0^T\!\!\int_\Omega
		\Big(\sum_{i=1}^{n} f_{i,\varepsilon}\Big)
		\frac{\delta}{\delta x_l}\!\boldsymbol{\psi}
		\\
		&=
		-\lim_{n\to\infty}
		\int_0^T\!\!\int_\Omega
		\Big(\sum_{i=1}^{n} \frac{\delta}{\delta x_l} f_{i,\varepsilon}\Big)
		\boldsymbol{\psi}
		\\
		&=
		-\int_0^T\!\!\int_\Omega
		\Big(\sum_{i=1}^{\infty} \frac{\delta}{\delta x_l} f_{i,\varepsilon}\Big)
		\boldsymbol{\psi}.
	\end{align*}
	Thus,
	\[
	\frac{\delta}{\delta x_l}\!\left(\sum_{i=1}^{\infty} f_{i,\varepsilon}\right)
	=
	\sum_{i=1}^{\infty} \frac{\delta}{\delta x_l} f_{i,\varepsilon}, \quad \mbox{in}\ (\delta,T)\times\Omega, \ \  \forall\, l=1,\cdots,N.
	\]
	Similarly, for the second derivative also, we have that
	\[
	\frac{\delta^2}{\delta x_r\delta x_l}\!\left(\sum_{i=1}^{\infty} f_{i,\varepsilon}\right)
	=
	\sum_{i=1}^{\infty} \frac{\delta^2}{\delta x_r\delta x_l} f_{i,\varepsilon}, \quad \mbox{in}\ (\delta,T)\times\Omega, \ \  \forall\, l,r=1,\cdots,N.
	\]

	
	

	\subsection{Differential inequality satisfied by the truncation}
	
	Using the chain rule and the properties of $T_m$, one obtains, in the sense of
	distributions,
	\begin{align}\label{differential inequality weak form}
		\partial_t v^m_{i,\varepsilon}
		- d_i \Delta v^m_{i,\varepsilon}
		\;\ge\;
		\Lambda^m_{i,\varepsilon}
		+
		\eta\Gamma^m_{i,\varepsilon},
	\end{align}
	where
	\[
	\Lambda^m_{i,\varepsilon}
	=
	T_m'(\omega_{i,\varepsilon})
	\left(
	Q_{i,\varepsilon}
	+ \eta \sum_{j\ne i} Q_{j,\varepsilon}
	\right),
	\]
	and
	\[
	\Gamma^m_{i,\varepsilon}
	=
	T_m'(\omega_{i,\varepsilon})
	\sum_{j\ne i} (d_j-d_i)\, \Delta f_{j,\varepsilon}.
	\]
	
	\subsection{Weak formulations of the truncated inequality}
	Take the time derivative and second order space derivative of $	v^m_{i,\varepsilon} = T_m(\omega_{i,\varepsilon})$ to get 
	\[
	\partial_t v^m_{i,\varepsilon}=
	T_m'(\omega_{i,\varepsilon}) \partial_t \omega_{i,\varepsilon}
	= T_m'(\omega_{i,\varepsilon})
	\left( 
	\partial_t f_{i,\varepsilon}
	+ \eta \sum_{j\ne i} \partial_t f_{j,\varepsilon}
	\right)
	\]
	and
	\[
	\Delta v^m_{i,\varepsilon}
	= T_m''(\omega_{i,\varepsilon})
	\left| 
	\nabla \omega_{i,\varepsilon}
	\right|^2
	+ T_m'(\omega_{i,\varepsilon})
	\left( 
	\Delta w_{i,\varepsilon}
	\right).
	\]
	We use the fact $T_m''\le 0$, it yields
	\[
	\Delta v^m_{i,\varepsilon}
	\;\le\;
	T_m'(\omega_{i,\varepsilon})
	\left( 
	\Delta f_{i,\varepsilon}
	+ \eta \sum_{j\ne i} \Delta f_{j,\varepsilon}
	\right).
	\]
	Now, consider the differential inequality
	\[
	\partial_t v^m_{i,\varepsilon}
	- d_i \Delta v^m_{i,\varepsilon}
	\;\ge\;
	\Lambda^m_{i,\varepsilon}
	+
	\eta\Gamma^m_{i,\varepsilon}
	\quad \text{in the distributional space}\  \mathcal{D}'((\delta,T)\times\Omega),
	\]
	where
	\[
	\Lambda^m_{i,\varepsilon}
	=
	T_m'(\omega_{i,\varepsilon})
	\left(
	Q_{i,\varepsilon}
	+ \eta \sum_{j\ne i} Q_{j,\varepsilon}
	\right),
	\qquad
	\Gamma^{m}_{i,\varepsilon}
	=
	T_m'(\omega_{i,\varepsilon})
	\sum_{j\ne i} (d_j - d_i)\, \Delta f_{j,\varepsilon},
	\]
	and
	\[
	\omega_{i,\varepsilon}
	=
	f_{i,\varepsilon}
	+
	\eta \sum_{k\ne i} f_{k,\varepsilon}.
	\]
	\begin{defn}[Weak formulation of supersolution to \eqref{differential inequality weak form}]
		Let 
		\[
		\psi \in C_c^\infty([0,T)\times\Omega),
		\qquad \psi \ge 0 .
		\]
		The above inequality is satisfied in the weak sense if
		\begin{align}
			\label{eq:WF1}
			- \int_{\Omega} \psi(\delta,x)\, &v^{m}_{i,\varepsilon}(\delta,x)\, dx
			+\int_{\delta}^{T} \int_{\Omega}
			(-\partial_t \psi)\, v^{m}_{i,\varepsilon}\, dx\, dt
			\nonumber+\, d_i \int_{\delta}^{T} \int_{\Omega}
			\nabla \psi \cdot \nabla v^{m}_{i,\varepsilon}\, dx\, dt
			\nonumber\\
			&\ge
			\int_{\delta}^{T} \int_{\Omega}
			\psi\, \Lambda^{m}_{i,\varepsilon}\, dx\, dt
			-\eta \sum_{j\ne i} (d_j - d_i)
			\int_{\delta}^{T} \int_{\Omega}
			T_m'(\omega_{i,\varepsilon})
			\,\nabla \psi \cdot \nabla f_{j,\varepsilon}
			\, dx\, dt
			\nonumber\\
			+\eta& \sum_{j\ne i} (d_j - d_i)
			\int_{\delta}^{T} \int_{\Omega}
			\psi\, T_m''(\omega_{i,\varepsilon})
			\left(
			\nabla f_{i,\varepsilon}
			+
			\eta \sum_{k\ne i} \nabla f_{k,\varepsilon}
			\right)
			\cdot \nabla f_{j,\varepsilon}
			\, dx\, dt .
		\end{align}
	\end{defn}
	\medskip
	
	\subsection{Passing to the limit \texorpdfstring{$\delta \to 0$}{delta to zero}}

	Let $m>0$ and $\varepsilon>0$ be fixed. We justify the passage to the limit
	$\delta\to0$ in the weak formulation \eqref{eq:WF1}, term by term. We start with the following two lemmas.
	\begin{lem}[Truncation energy estimate]\label{lem:trunc-energy}
		Let $d>0$, let $\Theta \in L^{1}(Q_T)$, and let $\mathcal{F}_{\rm{in}} \in L^{1}(\Omega)$.
		Let $\mathcal{F}$ be the solution of the system
		\begin{equation}\label{eq:trunc-heat}
			\begin{cases}
				\partial_t \mathcal{F} - d \Delta \mathcal{F} = \Theta & \text{in } (0,T)\times \Omega \\[2mm]
				\nabla \mathcal{F} \cdot \nu = 0, & \text{on } (0,T)\times\partial\Omega \\[2mm]
				\mathcal{F}(0,\cdot) = \mathcal{F}_{\rm{in}}, & \text{in } \Omega.
			\end{cases}
		\end{equation}
		Then, for every $M>0$, the following estimate holds:
		\begin{equation}\label{eq:trunc-energy}
			d \int_{\{|\mathcal{F}|\le M\}} |\nabla w|^2
			\le
			M\left(
			\int_{(0,T)\times\Omega} |\Theta|
			+
			\int_{\Omega} |\mathcal{F}_{\rm{in}}|
			\right).
		\end{equation}
	\end{lem}
	The proof of the lemma can be found in \cite{Pierre2010}.
	\begin{lem}[] \label{Root bound on NR operator}
		Let $T > 0$. Let $Q_{i,\varepsilon}$ denote the source term in \eqref{eq:RNFE} for all
		$i \in \mathbb{N}$ and for all $\varepsilon > 0$ on $(0,T)\times\Omega$. Furthermore, let the initial condition and diffusion coefficients satisfy
		\[
		\sum_{i=1}^{\infty} i f_i^{\mathrm{in}} \in L^2(\Omega),
		\qquad
		\sup_{i\in\mathbb{N}^*} d_i < +\infty.
		\] 
		Then the following estimate holds:
		\begin{align*}
			\sum_{i=1}^\infty \sqrt{\frac{\left\|Q_{i,\varepsilon}\right\|_{L^{1}((0,T)\times \Omega)}}{d_i}} <C,
		\end{align*}
		where the constant $C>0$ is independent of $\varepsilon$. 
	\end{lem}
	\begin{proof}
		We use the structure of $Q_{i,\varepsilon}$ as in \eqref{eq:RNFE}. It yields
		\begin{align*}
			\int_{0}^{T} \int_{\Omega} |Q_{i,\varepsilon}| \, dx\, dt \lesssim&
			\sum_{j=i+1}^{\infty} \sum_{k=1}^{j-1}
			\frac{b^{\,i}_{\,j-k,k}\, a_{\,j-k,k}}{k(j-k)d_{\,j-k}}
			\int_{0}^{T} \int_{\Omega} d_{\,j-k}\,(j-k) f_{\,j-k,\varepsilon}\, kf_{\,k,\varepsilon}\, dx\, dt\\
			&+\;
			\sum_{j=1}^{\infty}
			\frac{a_{i j}}{i j d_{j}}
			\int_{0}^{T} \int_{\Omega} d_{j}\, if_{i,\varepsilon}\, j f_{j,\varepsilon}\, dx\, dt.
		\end{align*}
		Thanks to the Proposition \ref{prop:dual_estimate}, we have that 
		\[
		\left( \int_{0}^{T} \int_{\Omega} |Q_{i,\varepsilon}| \, dx\, dt \right)^{1/2}
		\lesssim \left(
		\sum_{j=i+1}^{\infty} \sum_{k=1}^{j-1}
		\frac{\sqrt{\, b^{\,i}_{\,j-k,k}\, a_{\,j-k,k}}}{\sqrt{k (j-k)d_{\,j-k}}}
		\right)
		+
		\sum_{j=1}^{\infty}
		\frac{\sqrt{a_{i j}}}{\sqrt{i j d_j}} .
		\]
		
		\begin{align*}
			\sum_{i=1}^{\infty}\frac{1}{\sqrt{d_i}}
			\left( 
			\int_{0}^{T} \int_{\Omega} |Q_{i,\varepsilon}| \, dx\,dt 
			\right)^{1/2}
			&\lesssim 
			\sum_{i=1}^{\infty}
			\sum_{j=i+1}^{\infty}
			\sum_{k=1}^{j-1}
			\frac{\sqrt{\, b^{\,i}_{\,j-k,k}\, a_{\,j-k,k}}}{\sqrt{k(j-k)d_id_{\,j-k}}}
			+
			\sum_{i=1}^{\infty}\sum_{j=1}^{\infty} 
			\frac{\sqrt{a_{i j}}}{\sqrt{ij d_id_j}}\\
			&\lesssim	\sum_{j=1}^{\infty}
			\sum_{k=1}^{\infty}
			\sum_{i=1}^{j+k-1}
			\frac{\sqrt{\, b^{\,i}_{\,j,k}\, a_{\,j,k}}}{\sqrt{kjd_id_{\,j}}}
			+
			\sum_{i=1}^{\infty}\sum_{j=1}^{\infty} 
			\frac{\sqrt{a_{i j}}}{\sqrt{ijd_id_j}}.
		\end{align*}    
		Thanks to Assumption~\ref{strong assumption}, the right-hand side is finite. This concludes the proof.
	\end{proof}
	Using these two lemmas, we can analyze the time-limit $\delta \to 0$. More precisely, we have the following proposition.
	\begin{prop}\label{control by time parameter}
		Let $m>0$ and let $T_m$ is a smooth function as defined in \eqref{m_trunction}. Let $f_{i,\varepsilon}$ be a solution to~\eqref{eq:RNFE} for $i \in \mathbb{N}$ and
		$\varepsilon > 0$. Then, for every nonnegative test function
		$\psi \in C_c^\infty([0,T)\times\Omega)$, there exists a constant $C > 0$, independent of
		$\varepsilon$, $\eta$, and the index $i$, such that
		\begin{align}\label{weak formulation on time}
			\int_{\Omega}\psi(0,x)\,T_m(\omega_i^{\mathrm{in}})\,dx
			+&
			\int_{0}^{T}\!\!\int_{\Omega}
			(-\partial_t\psi)\,T_m(\omega_{i,\varepsilon})\,dx\,dt
			+
			d_i\int_{0}^{T}\!\!\int_{\Omega}
			\nabla\psi\cdot\nabla T_m(\omega_{i,\varepsilon})\,dx\,dt \nonumber
			\\
			&\qquad\ge
			\int_{0}^{T}\!\!\int_{\Omega}
			\psi\,\Lambda^{m}_{i,\varepsilon}\,dx\,dt
			-
			C\,m\,\eta^{1/2},
		\end{align}
		where
		\[
		\omega_{i,\varepsilon}
		=
		f_{i,\varepsilon}
		+
		\eta \sum_{j\ne i} f_{j,\varepsilon},
		\ \text{and} \ 
		\Lambda^{m}_{i,\varepsilon}
		=
		T_m'(\omega_{i,\varepsilon})
		\Bigl(
		Q_{i,\varepsilon}
		+
		\eta \sum_{j\ne i} Q_{j,\varepsilon}
		\Bigr).
		\]
	\end{prop}
	\begin{proof}
		We estimate the limit $\delta\to 0$ term by term. 
		\newline
		\textbf{First term (initial condition).}
		For each $i\ge1$, the mild solution $f_{i,\varepsilon}$ to \eqref{eq:RNFE} satisfies
		\[
		f_{i,\varepsilon}(t) \longrightarrow f_i^{\mathrm{in}}
		\qquad \text{in } L^1(\Omega) \text{ as } t\to0 .
		\]
		Next, we derive that $\displaystyle{\sum\limits_{i=1}^{\infty}f_{i,\varepsilon}(t) \longrightarrow \sum\limits_{i=1}^{\infty}f_i^{\mathrm{in}}}$ in  $L^1(\Omega)$ as  $t\to0 $. The solution $f_{i,\varepsilon}$ to \eqref{eq:RNFE} can be expressed as:
		\begin{align}\label{representation regularization}
			f_{i,\varepsilon}(t,x)= \int_{\Omega}G_{d_i}(t,0,x,y)f_i^{\rm{in}}(y)\, dy+\int_{0}^{t}\int_{\Omega} G_{d
				_i}(t,s,x,y)Q_{i,\varepsilon}(s,y) \, ds\, dy, \ \ \forall \, i\in\mathbb{N},
		\end{align}
		where $G_{d_i}(t,s,x,y)$ is the Neumann Green function corresponding to the operator $\partial_t-d_i\Delta$. Hence, we obtain the following $L^1(\Omega)$ integral estimate
		\begin{align*}
			\sum_{i>M} \|f_{i,\varepsilon}(t,x)\|_{L^1(\Omega)} \leq &\sum_{i>M} \| f_i^{\rm{in}}\|_{L^1(\Omega)}+ \sum_{i>M} \left\| \int_{0}^{t}\int_{\Omega} G(t,s,x,y)Q_{i,\varepsilon}(s,y) \, ds\, dy\right\|_{L^1(\Omega)}\\
			\leq & \sum_{i>M} \| f_i^{\rm{in}}\|_{L^1(\Omega)}+ \sum_{i>M} \int_{0}^{t} \| Q_{i,\varepsilon}(s,\cdot)\|_{L^1(\Omega)} \, ds\\
			\leq &\sum_{i>M} \| f_i^{\rm{in}}\|_{L^1(\Omega)}+\sum_{i>M} \| Q_{i,\varepsilon}\|_{L^1((0,T)\times\Omega)} \leq \delta_M\to 0 \ \mbox{as}\ M\to \infty.
		\end{align*}
		Here, thanks to Lemma~\eqref{bound on NR operator}, we obtain the last line. Now
		\begin{align*}
			&\limsup\limits_{t\to 0}\left\| \sum_{i=1}^{\infty} f_{i,\varepsilon}(t,x)- \sum_{i=1}^{\infty} f_{i}^{\rm{in}}(x)\right\|_{L^1(\Omega)} \\
			\leq & \limsup\limits_{t\to 0}\left\| \sum_{i=1}^{M} f_{i,\varepsilon}(t,x)- \sum_{i=1}^{M} f_{i}^{\rm{in}}(x)\right\|_{L^1(\Omega)}\\
			&+ \limsup\limits_{t\to 0}\sum_{i>M} \|f_{i,\varepsilon}(t,x)\|_{L^1(\Omega)} +\limsup\limits_{t\to 0}\sum_{i>M} \| f_i^{\rm{in}}\|_{L^1(\Omega)}\\
			\leq & 2\delta_M \to 0 \ \ \mbox{as} \ M\to +\infty. 
		\end{align*}
		Hence $\displaystyle{\sum_{i=1}^{\infty} f_{i,\varepsilon}(t,x)\to \sum_{i=1}^{\infty} f_{i}^{\rm{in}}(x) }$ in $L^1(\Omega)$. This further implies $\displaystyle{\sum_{i=1}^{\infty} f_{i,\varepsilon}(t,x)\to \sum_{i=1}^{\infty} f_{i}^{\rm{in}}(x) }$ for all $x\in \Omega$ pointwise a.e.
		\newline
		Since $T_m$ is continuous and bounded and $\psi$ is smooth, we have
		the following pointwise a.e. convergence
		\[
		v^m_{i,\varepsilon}(t,x)
		\longrightarrow
		T_m\!\left(
		f_i^{\mathrm{in}}(x)
		+ \eta \sum_{j\ne i} f_j^{\mathrm{in}}(x)
		\right)
		\quad \text{as } t\to0 .
		\]
		Moreover,
		\(
		0 \le v^m_{i,\varepsilon} \le m
		\),
		hence employing dominated convergence theorem, we obtain that
		\[
		\lim_{\delta\to0}
		\int_{\Omega} \psi(\delta,x)\, v^m_{i,\varepsilon}(\delta,x)\,dx
		=
		\int_{\Omega}
		\psi(0,x)\,
		T_m\!\left(
		f_i^{\mathrm{in}}(x)
		+ \eta \sum_{j\ne i} f_j^{\mathrm{in}}(x)
		\right) dx .
		\]
		
		\medskip
		\noindent\textbf{Second term (time derivative).}
		By construction of the truncation,
		\[
		0 \le v^m_{i,\varepsilon} \le m
		\quad \text{a.e. in } (0,T)\times\Omega,
		\]
		and therefore
		\(
		v^m_{i,\varepsilon} \in L^1((0,T)\times\Omega).
		\)
		Since $\partial_t\psi \in L^\infty((0,T)\times\Omega)$, we can pass the
		limit by dominated convergence theorem:
		\[
		\lim_{\delta\to0}
		\int_{\delta}^{T}\!\!\int_{\Omega}
		(\partial_t\psi)\, v^m_{i,\varepsilon}\,dx\,dt
		=
		\int_{0}^{T}\!\!\int_{\Omega}
		(\partial_t\psi)\, v^m_{i,\varepsilon}\,dx\,dt .
		\]
		
		\medskip
		\noindent\textbf{Third term (diffusion).}
		Using the chain rule,
		\[
		\nabla v^m_{i,\varepsilon}
		=
		T_m'(\omega_{i,\varepsilon})\,\nabla\omega_{i,\varepsilon},
		\qquad
		0 \le T_m' \le 1,
		\]
		where
		\(
		\omega_{i,\varepsilon}
		=
		f_{i,\varepsilon}
		+ \eta \sum_{j\ne i} f_{j,\varepsilon}.
		\) For each $i$, the solution $f_{i\varepsilon}$ to \eqref{eq:RNFE} can be expressed as (see \eqref{representation regularization}):
		\[
		f_{i,\varepsilon}(t,x)= \int_{\Omega}G_{d_i}(t,0,x,y)f_i^{\rm{in}}(y)\, dy+\int_{0}^{t}\int_{\Omega} G_{d
			_i}(t,s,x,y)Q_{i,\varepsilon}(s,y) \, ds\, dy, \ \ \forall \, i\in\mathbb{N}.
		\]
		Following step 5 in the Proposition~\eqref{prop:infinite_mass}, we can derive the following integral estimate of the gradient
		\begin{align}\label{gradient estimate:regularization}
			\|\nabla f_{i,\varepsilon}\|_{L^1((0,T)\times\Omega)}
			\le
			\frac{C\sqrt{T}}{\sqrt{d_i}}
			\left(
			\|f_i^{\mathrm{in}}\|_{L^1(\Omega)}
			+
			\|Q_{i,\varepsilon}\|_{L^1((0,T)\times\Omega)}
			\right).
		\end{align}
		Thanks to Lemma \ref{bound on NR operator}, we conclude the following estimate
		\begin{align*}
			\int_0^T\!\!\int_\Omega
			|\nabla v^m_{i,\varepsilon}|
			\le &
			C\sqrt{T}
			\left(
			\frac{1}{\sqrt{d_i}}\|f_i^{\mathrm{in}}\|_{L^1(\Omega)}
			+
			\sum_{j\ne i}\frac{\eta}{\sqrt{d_j}}\|f_j^{\mathrm{in}}\|_{L^1(\Omega)}\right)\\
			&
			+
			C\sqrt{T}\sum_{j}\frac{\eta}{\sqrt{d_j}}\|Q_{j,\varepsilon}\|_{L^1((0,T)\times\Omega)}<+\infty,
		\end{align*}
		which implies
		\[
		\lim_{\delta\to0}
		d_i\int_{\delta}^{T}\!\!\int_\Omega
		\nabla\psi\cdot\nabla v^m_{i,\varepsilon}\,dx\,dt
		=
		d_i\int_{0}^{T}\!\!\int_\Omega
		\nabla\psi\cdot\nabla v^m_{i,\varepsilon}\,dx\,dt .
		\]
		
		\medskip
		\noindent\textbf{Fourth term.}
		Recall that
		\[
		\Lambda^m_{i,\varepsilon}
		=
		T_m'(\omega_{i,\varepsilon})
		\Bigl(
		Q_{i,\varepsilon}
		+
		\eta \sum_{j\ne i} Q_{j,\varepsilon}
		\Bigr),
		\qquad
		0 \le T_m' \le 1 .
		\]
		Hence,
		\[
		|\Lambda^m_{i,\varepsilon}|
		\le
		|Q_{i,\varepsilon}|
		+
		\eta \sum_{j\ne i} |Q_{j,\varepsilon}|
		\in L^1((0,T)\times\Omega).
		\]
		Dominated convergence theorem yields
		\[
		\lim_{\delta\to0}
		\int_{\delta}^{T}\!\!\int_\Omega
		\psi\,\Lambda^m_{i,\varepsilon}\,dx\,dt
		=
		\int_{0}^{T}\!\!\int_\Omega
		\psi\,\Lambda^m_{i,\varepsilon}\,dx\,dt .
		\]
		To deal with the fifth and sixth terms, Lemma \ref{lem:trunc-energy} and Lemma \ref{Root bound on NR operator} will be very useful.
		\medskip
		\noindent\textbf{Fifth and sixth terms.}
		Recall that
		\[
		\Gamma^{m}_{i,\varepsilon}
		=
		T_m'(\omega_{i,\varepsilon})
		\sum_{j\ne i} (d_j-d_i)\,\Delta f_{j,\varepsilon},
		\qquad
		\omega_{i,\varepsilon}
		=
		f_{i,\varepsilon}
		+
		\eta \sum_{k\ne i} f_{k,\varepsilon}.
		\]
		In the weak formulation \eqref{eq:WF1}, these terms already appear after
		integration by parts, namely,
		\begin{align*}
			\int_{\delta}^{T}\!\!\int_{\Omega}
			\psi\,\Gamma^{m}_{i,\varepsilon}
			&=
			-\sum_{j\ne i}(d_j-d_i)
			\int_{\delta}^{T}\!\!\int_{\Omega}
			T_m'(\omega_{i,\varepsilon})\,
			\nabla\psi\cdot\nabla f_{j,\varepsilon}\,dx\,dt
			\\
			&\quad
			+
			\sum_{j\ne i}(d_j-d_i)
			\int_{\delta}^{T}\!\!\int_{\Omega}
			\psi\,
			T_m''(\omega_{i,\varepsilon})
			\Bigl(
			\nabla f_{i,\varepsilon}
			+
			\eta \sum_{k\ne i}\nabla f_{k,\varepsilon}
			\Bigr)\cdot\nabla f_{j,\varepsilon}\,dx\,dt .
		\end{align*}
		\medskip
		\textbf{Estimate of the term involving \(T_m''\).}
		Since \(T_m''\) is supported in \(\{0\le \omega_{i,\varepsilon}\le m\}\) and
		\(|T_m''|\le 1\), we obtain
		\begin{align*}
			&\sum_{j=1}^{\infty}
			\left|
			\int_{\delta}^{T}\!\!\int_{\Omega}
			\psi\,
			T_m''(\omega_{i,\varepsilon})
			\Bigl(
			\nabla f_{i,\varepsilon}
			+
			\eta \sum_{k\ne i}\nabla f_{k,\varepsilon}
			\Bigr)\cdot\nabla f_{j,\varepsilon}\,dx\,dt
			\right|
			\\
			&\qquad
			\le
			C\sum_{j=1}^{\infty}
			\int_{\delta}^{T}\!\!\int_{\Omega_{ij}}
			\Bigl(
			|\nabla f_{i,\varepsilon}|\,|\nabla f_{j,\varepsilon}|
			+
			\eta\sum_{k\ne i}
			|\nabla f_{k,\varepsilon}|\,|\nabla f_{j,\varepsilon}|
			\Bigr)\,dx\,dt,
		\end{align*}
		where
		\[
		\Omega_{ij}
		=
		\Bigl\{
		(t,x)\in(\delta,T)\times\Omega:
		|f_{i,\varepsilon}|\le m,\;
		|f_{j,\varepsilon}|\le m/\eta,\;
		j\ne i
		\Bigr\}.
		\]
		Employing Cauchy-Schwarz inequality, we have that
		\begin{align*}
			&\le
			C\sum_{j=1}^{\infty}
			\left(
			\int_{\delta}^{T}\!\!\int_{\Omega_{ij}}
			|\nabla f_{i,\varepsilon}|^{2}
			\right)^{1/2}
			\left(
			\int_{\delta}^{T}\!\!\int_{\Omega_{ij}}
			|\nabla f_{j,\varepsilon}|^{2}
			\right)^{1/2}
			\\
			&\quad
			+
			C\eta
			\sum_{j=1}^{\infty}\sum_{k\ne i}
			\left(
			\int_{\delta}^{T}\!\!\int_{\Omega_{kj}}
			|\nabla f_{k,\varepsilon}|^{2}
			\right)^{1/2}
			\left(
			\int_{\delta}^{T}\!\!\int_{\Omega_{kj}}
			|\nabla f_{j,\varepsilon}|^{2}
			\right)^{1/2}.
		\end{align*}
		\medskip
		\textbf{Estimate of the gradient terms on truncation sets.}
		From the definition of \(\Omega_{ij}\),
		\[
		\int_{\Omega_{ij}} |\nabla f_{j,\varepsilon}|^{2}
		\le
		\int_{\{|f_{j,\varepsilon}|\le m/\eta\}}
		|\nabla f_{j,\varepsilon}|^{2}.
		\]
		Using Lemma \ref{lem:trunc-energy} to the triplet \( \big(\mathcal{F}=f_{j,\varepsilon}, \, M=m/\eta, \, \Theta=Q_{j,\varepsilon}, j\neq i\big)\) and to the triplet \( \big(w=f_{i,\varepsilon}, \, M=m, \, \Theta=Q_{i,\varepsilon}\big)\) respectively, we obtain
		\begin{align}\label{use_small_initial_mas}
			\left(
			\int_{\Omega_{ij}} |\nabla f_{j,\varepsilon}|^{2}
			\right)^{1/2}
			\le
			\left(\frac{m}{\eta d_i}\right)^{1/2}
			\left[
			\left(
			\int_{0}^{T}\!\!\int_{\Omega}|Q_{j,\varepsilon}|
			\right)^{1/2}
			+
			\left(
			\int_{\Omega} f_j^{\mathrm{in}}
			\right)^{1/2}
			\right],\\
			\left(
			\int_{\Omega_{ij}} |\nabla f_{i,\varepsilon}|^{2}
			\right)^{1/2}
			\le
			\left(\frac{m}{d_i}\right)^{1/2}
			\left[
			\left(
			\int_{0}^{T}\!\!\int_{\Omega}|Q_{i,\varepsilon}|
			\right)^{1/2}
			+
			\left(
			\int_{\Omega} f_i^{\mathrm{in}}
			\right)^{1/2}
			\right].
		\end{align}
		Using the estimates on the reaction terms \(Q_{j,\varepsilon}\) as in Lemma \ref{Root bound on NR operator}  and summing over the index \(j\), we finally obtain
		\[
		\left|
		\int_{\delta}^{T}\!\!\int_{\Omega}
		\psi\,
		T_m''(\omega_{i,\varepsilon})
		(\nabla f_{i,\varepsilon}
		+\eta\sum_{k\ne i}\nabla f_{k,\varepsilon})
		\cdot\nabla f_{j,\varepsilon}
		\,dx\,dt
		\right|
		\le
		C\,m\,\eta^{-1/2}.
		\]
		\textbf{Estimate of the term involving \(T_m'\nabla\psi\).}
		Since \(0\le T_m'\le1\) and \(\nabla\psi\in L^\infty\), a similar argument
		based on the $L^1$–integrability of \(\nabla f_{j,\varepsilon}\) yields
		\[
		\left|
		\sum_{j\ne i}(d_j-d_i)
		\int_{\delta}^{T}\!\!\int_{\Omega}
		T_m'(\omega_{i,\varepsilon})\,
		\nabla\psi\cdot\nabla f_{j,\varepsilon}\,dx\,dt
		\right|
		\le
		C\,m\,\eta^{-1/2}.
		\]
		\noindent\textbf{Final weak formulation.}
		Combining all the estimates above and letting $\delta \to 0$ in \eqref{eq:WF1}, we
		conclude that, for every nonnegative $\psi\in C_c^\infty([0,T)\times\Omega)$,
		\begin{align*}
			\int_{\Omega}\psi(0,x)\,T_m(\omega_i^{\mathrm{in}})\,dx
			+ &
			\int_{0}^{T}\!\!\int_{\Omega}
			(-\partial_t\psi)\,T_m(\omega_{i,\varepsilon})\,dx\,dt 
			+
			d_i\int_{0}^{T}\!\!\int_{\Omega}
			\nabla\psi\cdot\nabla T_m(\omega_{i,\varepsilon})\,dx\,dt
			\\
			&\qquad\ge
			\int_{0}^{T}\!\!\int_{\Omega}
			\psi\,\Lambda^{m}_{i,\varepsilon}\,dx\,dt
			-
			C\,m\,\eta^{1/2},
		\end{align*}
		where
		\[
		\omega_{i,\varepsilon}
		=
		f_{i,\varepsilon}
		+
		\eta \sum_{j\ne i} f_{j,\varepsilon},
		\ \text{and} \ 
		\Lambda^{m}_{i,\varepsilon}
		=
		T_m'(\omega_{i,\varepsilon})
		\Bigl(
		Q_{i,\varepsilon}
		+
		\eta \sum_{j\ne i} Q_{j,\varepsilon}
		\Bigr).
		\]
	\end{proof}
	
	\subsection{Passing to the limit \texorpdfstring{$\varepsilon \to 0$}{epsilon to zero}}

	We now pass to the limit $\varepsilon \to 0$ in the estimate \ref{weak formulation on time} as in Proposition \ref{control by time parameter}. Throughout this subsection, $m>0$ and $\eta>0$ are
	fixed.
	\begin{prop}\label{epsilon compactness in the inequation}
		Let $m>0$, and let $T_m$ be a smooth function as defined in \eqref{m_trunction}. Let $f_i$ 
		denote the $L^1((0,T); W^{1,1}(\Omega))$ limit of the sequence $\{f_{i,\varepsilon}\}$, 
		where $f_{i,\varepsilon}$ is a solution to the regularized nonlinear fragmentation system~\eqref{eq:RNFE} 
		for all $i \in \mathbb{N}$ and for all $\varepsilon > 0$. Further assume that
		\[
		\sum_{i=1}^{\infty} i f_i^{\mathrm{in}} \in L^2(\Omega),
		\ \text{and} \ 
		\sup_{i\in\mathbb{N}^*} d_i < +\infty.
		\]
		Then, for every nonnegative test function 
		$\psi \in C_c^\infty([0,T)\times\Omega)$, we have
		\begin{align}\label{control estimate on epsilon}
			\int_{\Omega} \psi(0,x)\,T_m(\omega_i^{\mathrm{in}})\,dx
			+&
			\int_{0}^{T}\!\!\int_{\Omega} (-\partial_t\psi)\,T_m(\omega_i)\,dx\,dt
			+
			d_i \int_{0}^{T}\!\!\int_{\Omega} \nabla\psi \, T_m'(\omega_i) \, \nabla\omega_i \,dx\,dt 
			\nonumber  \\
			&\ge
			\int_{0}^{T}\!\!\int_{\Omega} \psi\,\Lambda^m_i\,dx\,dt
			-
			C\,m^2(1+T)^2\,\eta^{1/2},
		\end{align}
		where
		\[
		\omega_i = f_i + \eta \sum_{j\ne i} f_j, 
		\ \text{and} \ 
		\Lambda^m_i = T_m'(\omega_i) \Bigl( Q_i + \eta \sum_{j\ne i} Q_j \Bigr).
		\]
		Here $Q_i$ is as defined in \eqref{NFE Operator}.
	\end{prop}
	\begin{proof}
		We pass the limit $\varepsilon\to 0$ term by term.
		\newline
		\textbf{Convergence of the truncated densities.}
		We consider
		\[
		\int_{0}^{T}\!\!\int_{\Omega}
		\left|
		f_{i,\varepsilon}
		+ \eta \sum_{j\ne i} f_{j,\varepsilon}
		-
		f_i
		- \eta \sum_{j\ne i} f_j
		\right| dx\,dt .
		\]
		Using the triangle inequality, we obtain
		\begin{align}\label{intermediate 1}
			\le
			\int_{0}^{T}\!\!\int_{\Omega}
			|f_{i,\varepsilon}-f_i|\,dx\,dt
			+
			\eta \sum_{j\ne i}
			\int_{0}^{T}\!\!\int_{\Omega}
			|f_{j,\varepsilon}-f_j|\,dx\,dt .
		\end{align}
		Fix $M\in\mathbb{N}$. Recall that the solution $f_{i,\varepsilon}$ can be expressed as \eqref{representation regularization}:
		\[
		f_{i,\varepsilon}(t,x)= \int_{\Omega}G(t,0,x,y)f_i^{\rm{in}}(y)\, dy+\int_{0}^{t}\int_{\Omega} G(t,s,x,y)Q_{i,\varepsilon}(s,y) \, ds\, dy, \ \ \forall \, i\in\mathbb{N}.
		\]
		Hence, we obtain the following $L^1((0,T)\times\Omega)$ integral estimate
		\begin{align*}
			\sum_{i>M} \|f_{i,\varepsilon}(t,x)\|_{L^1((0,T)\times\Omega)} \leq &\sum_{i>M} \int_0^T\| f_i^{\rm{in}}\|_{L^1(\Omega)}+ \sum_{i>M} \left\| \int_{0}^{t}\int_{\Omega} G(t,s,x,y)Q_{i,\varepsilon}(s,y) \, ds\, dy\right\|_{L^1((0,T)\times\Omega)}\\
			\leq &T\sum_{i>M} \| f_i^{\rm{in}}\|_{L^1(\Omega)}+\sum_{i>M} \| Q_{i,\varepsilon}\|_{L^1((0,T)\times\Omega)} \\\leq & T\sum_{i>M} \| f_i^{\rm{in}}\|_{L^1(\Omega)}+C\sum_{i>M}
			\sum_{j=i+1}^{\infty} \sum_{k=1}^{j-1}
			\frac{b_{j-k,k}^i a_{j-k,k}}{k(j-k)d_{k}d_i}+C\sum_{i>M}
			\sum_{j=1}^{\infty}
			\frac{a_{i,j}}{ijd_jd_i} \\
			\leq & \delta_M\to 0 \ \mbox{as}\ M\to \infty.
		\end{align*}
		Here the third line is due to Proposition \ref{prop:dual_estimate}. Since $f_{i,\varepsilon}\ge0$ and $f_{i,\varepsilon}\to f_i$ pointwise a.e.,
		Fatou's lemma yields
		\[
		\sum_{i>M} \|f_{i}\|_{L^1((0,T)\times\Omega)} \leq \sum_{i>M} \|f_{i,\varepsilon}\|_{L^1((0,T)\times\Omega)}\leq \delta_M\to 0 \ \mbox{as}\ M\to \infty.
		\]
		Splitting the sum in \eqref{intermediate 1} into a finite and an infinite part,
		we obtain
		\[
		\le
		\int_{0}^{T}\!\!\int_{\Omega}
		|f_{i,\varepsilon}-f_i|\,dx\,dt
		+
		\eta \sum_{j\ne i}^{M}
		\int_{0}^{T}\!\!\int_{\Omega}
		|f_{j,\varepsilon}-f_j|\,dx\,dt
		+
		\eta\,\delta_M,
		\]
		where $\delta_M\to0$ as $M\to\infty$, uniformly in $\varepsilon$, thanks to
		the uniform $L^1$ summability of $\{f_{j,\varepsilon}\}_{j\ge1}$.
		Since $f_{j,\varepsilon}\to f_j$ strongly in $L^1((0,T)\times\Omega)$ for
		each fixed $j$, we conclude that
		\[
		f_{i,\varepsilon}
		+ \eta \sum_{j\ne i} f_{j,\varepsilon}
		\;\longrightarrow\;
		f_i + \eta \sum_{j\ne i} f_j
		\quad \text{in } L^1((0,T)\times\Omega).
		\]
		In particular,
		\[
		\omega_{i,\varepsilon}
		\longrightarrow
		\omega_i
		\quad \text{a.e. and in } L^1((0,T)\times\Omega).
		\]
		
		\medskip
		\noindent\textbf{Limit of the second term (time derivative).}
		Since $T_m$ is continuous and bounded, we have
		\[
		T_m(\omega_{i,\varepsilon}) \to T_m(\omega_i)
		\quad \text{a.e. in } (0,T)\times\Omega,
		\qquad
		|T_m(\omega_{i,\varepsilon})|\le m .
		\]
		Therefore, by the Dominated Convergence Theorem,
		\[
		\lim_{\varepsilon\to0}
		\int_{0}^{T}\!\!\int_{\Omega}
		(-\partial_t\psi)\,T_m(\omega_{i,\varepsilon})\,dx\,dt
		=
		\int_{0}^{T}\!\!\int_{\Omega}
		(-\partial_t\psi)\,T_m(\omega_i)\,dx\,dt .
		\]
		
		\medskip
		\noindent\textbf{Gradient tail estimation.}
		We recall the following a priori estimates on the gradient term derived earlier \eqref{gradient estimate:regularization}, 
		\[
		\|\nabla f_{i,\varepsilon}\|_{L^1((0,T)\times\Omega)}
		\le
		C\sqrt{T}\,\frac{1}{\sqrt{d_i}}
		\left(
		\|f_i^{\mathrm{in}}\|_{L^1(\Omega)}
		+
		\|Q_{i,\varepsilon}\|_{L^1((0,T)\times\Omega)}
		\right).
		\]
		The sum of the tail can be estimated as 
		\begin{align}\label{gradient uniform sum}
			\sum_{i>M}
			\|\nabla f_{i,\varepsilon}\|_{L^1((0,T)\times\Omega)}
			\le C\sqrt{T}\, \sum_{i>M}\frac{1}{\sqrt{d_i}}
			\left(
			\|f_i^{\mathrm{in}}\|_{L^1(\Omega)}
			+
			\sum_{i>M}\|Q_{i,\varepsilon}\|_{L^1((0,T)\times\Omega)}
			\right)
		\end{align}
		Therefore,
		\begin{align*}
			\left\|
			\frac{Q_{i,\varepsilon}(f)}{d_i}
			\right\|_{L^1((0,T)\times\Omega)}
			&\le
			\frac{1}{2}\left\|
			\sum_{j=i+1}^{\infty} \sum_{k=1}^{j-1}
			\frac{b_{j-k,k}^i a_{j-k,k}}{k(j-k)d_{k}d_i}\,
			(j-k)f_{j-k,\varepsilon} kd_kf_{k,\varepsilon}
			\right\|_{L^1((0,T)\times\Omega)}
			\\
			&\quad+
			\left\|
			\sum_{j=1}^{\infty}
			\frac{a_{i,j}}{ijd_jd_i}\, if_{i,\varepsilon} jd_jf_{j,\varepsilon}
			\right\|_{L^1((0,T)\times\Omega)}.\\
			\lesssim & \frac{1}{2}
			\sum_{j=i+1}^{\infty} \sum_{k=1}^{j-1}
			\frac{b_{j-k,k}^i a_{j-k,k}}{k(j-k)d_{k}d_i}\,
			+ \sum_{j=1}^{\infty}
			\frac{a_{i,j}}{ijd_jd_i},
		\end{align*}
		where we obtain the last inequality thanks to Proposition~\ref{prop:dual_estimate}. Hence the tail 
		\begin{align}\label{tail control source}
			{\sum_{i>M}\left\|
				\frac{Q_{i,\varepsilon}(f)}{d_i}
				\right\|_{L^1((0,T)\times\Omega)}\le \delta_M \to 0}
		\end{align}
		as $M\to\infty$. As a consequence, for $\delta_M>0$, there exists $M\in\mathbb{N}$ such that
		\[
		\sum_{j>M}
		\|\nabla f_{j,\varepsilon}\|_{L^1((0,T)\times\Omega)}
		<\delta_M \to 0 \ \ \text{as}\ M\to+\infty,
		\ \text{uniformly in } \varepsilon.
		\]
		Since $\nabla f_{i,\varepsilon}\to \nabla f_{i}$ pointwise. Fatou's lemma yields
		\[
		\sum_{j>M}
		\|\nabla f_{j}\|_{L^1((0,T)\times\Omega)}
		<\delta_M \to 0 \ \ \text{as}\ M\to+\infty.
		\]
		\medskip
		\noindent\textbf{Convergence of the gradients of $\omega_{i,\varepsilon}$.}
		We consider
		\[
		\int_{0}^{T}\!\!\int_{\Omega}
		\left|
		\nabla f_{i,\varepsilon}
		+ \eta \sum_{j\ne i} \nabla f_{j,\varepsilon}
		-
		\nabla f_i
		-
		\eta \sum_{j\ne i} \nabla f_j
		\right| dx\,dt .
		\]
		Using the triangle inequality and splitting the infinite sum at $M$, we obtain
		\begin{align*}
			\le{}&
			\int_{0}^{T}\!\!\int_{\Omega}
			|\nabla f_{i,\varepsilon}-\nabla f_i|\,dx\,dt
			+
			\eta \sum_{\substack{j=1\\ j\ne i}}^{M}
			\int_{0}^{T}\!\!\int_{\Omega}
			|\nabla f_{j,\varepsilon}-\nabla f_j|\,dx\,dt
			\\
			&\qquad
			+
			\eta \sum_{j>M}
			\left(
			\|\nabla f_{j,\varepsilon}\|_{L^1}
			+
			\|\nabla f_j\|_{L^1}
			\right).
		\end{align*}
		For each fixed $j$, we have $\nabla f_{j,\varepsilon}\to\nabla f_j$ in
		$L^1((0,T)\times\Omega)$, while the tail is controlled uniformly by the
		previous estimate. Hence,
		\[
		\nabla\omega_{i,\varepsilon}
		\longrightarrow
		\nabla\omega_i
		\quad \text{in } L^1((0,T)\times\Omega).
		\]
		
		\medskip
		\noindent\textbf{Limit of the third term (diffusion).}
		Since $T_m'$ is bounded and continuous, it follows that
		\[
		T_m'(\omega_{i,\varepsilon})\,\nabla\omega_{i,\varepsilon}
		\longrightarrow
		T_m'(\omega_i)\,\nabla\omega_i
		\quad \text{in } L^1((0,T)\times\Omega).
		\]
		Therefore,
		\[
		\lim_{\varepsilon\to0}
		d_i \int_{0}^{T}\!\!\int_{\Omega}
		\nabla\psi\,
		T_m'(\omega_{i,\varepsilon})\,
		\nabla\omega_{i,\varepsilon}\,dx\,dt
		=
		d_i \int_{0}^{T}\!\!\int_{\Omega}
		\nabla\psi\,
		T_m'(\omega_i)\,
		\nabla\omega_i\,dx\,dt .
		\]
		
		\medskip
		\noindent\textbf{Limit of the fourth term (reaction).}
		Recall that
		\[
		\Lambda^m_{i,\varepsilon}
		=
		T_m'(\omega_{i,\varepsilon})
		\left(
		Q_{i,\varepsilon}
		+
		\eta \sum_{j\ne i} Q_{j,\varepsilon}
		\right).
		\]
		Fix $(t,x)\in \text{Supp}\, \{ T'_m{(w_{j,\varepsilon})}\}$.  
		By the construction of the truncation $T_m$ \eqref{m_trunction}, we have, for all $j\ge1$ and
		$\varepsilon>0$,
		\[
		0 \le f_{j,\varepsilon}(t,x) \le m .
		\]
		Consequently, 
		\[
		0 \le \varepsilon\sum_{j=1}^{\infty} c_j (f_{j,\varepsilon}(t,x))^2
		\le
		\varepsilon m^2 \sum_{j=1}^{\infty} c_j, \quad \forall \, (t,x)\in \text{Supp}\, \{ T'_m{(w_{j,\varepsilon})}\}. 
		\]
		By the assumption on the coefficients (see assumption~\ref{strong assumption}), we have
		\[
		\sum_{j=1}^{\infty} c_j < \infty .
		\]
		Hence,
		\[
		\varepsilon m^2 \sum_{j=1}^{\infty} c_j
		\;\xrightarrow[\varepsilon\to0]{}\; 0 .
		\]
		The above estimate is pointwise and uniform with respect to $(t,x)\in \text{Supp}\, \{ T'_m{(w_{j,\varepsilon})}\}$.
		As a consequence,
		\[
		1+\varepsilon\sum_{j=1}^{\infty} c_j (f_{j,\varepsilon})^2(t,x)
		\xrightarrow[\varepsilon\to0]{}
		1,
		\quad \forall \, (t,x)\in \text{Supp}\, \{ T'_m{(w_{j,\varepsilon})}\} .
		\]
		Similarly, with the help of uniform tail control \eqref{tail control source} on $Q_{j,\varepsilon}$, for any $j\ge 1$, we obtain
		\[
		T_{m}^{'}(w_{j,\varepsilon})Q_{j,\varepsilon} \to T_m^{'}(w_j)Q_j
		\]
		It yields
		\[
		\Lambda^m_{i,\varepsilon} \to \Lambda^m_i
		:=
		T_m'(\omega_i)
		\left(
		Q_i + \eta \sum_{j\ne i} Q_j
		\right).
		\]
		Moreover, the truncation yields the domination
		\[
		|\psi\,\Lambda^m_{i,\varepsilon}|
		\le
		\left(\frac{m}{\eta}\right)^2
		\left(
		\sum_{i=1}^{\infty}\sum_{j=i+1}^{\infty}\sum_{k=1}^{j-1}
		b^i_{j-k,k}\,a_{j-k,k}
		+
		\sum_{i=1}^{\infty}\sum_{j=1}^{\infty} a_{ij}
		\right),
		\]
		which is integrable. Hence, by the Dominated Convergence Theorem,
		\[
		\Lambda^m_{i,\varepsilon} \to \Lambda^m_i
		\quad \text{in } L^1((0,T)\times\Omega).
		\]
		
		\medskip
		\noindent\textbf{Limit weak formulation.}
		Letting $\varepsilon\to0$ in the weak formulation, we conclude that for every
		nonnegative test function $\psi\in C_c^\infty([0,T)\times\Omega)$,
		\begin{align*}
			\int_{\Omega} \psi(0,x)\,T_m(\omega_i^{\mathrm{in}})\,dx
			+&
			\int_{0}^{T}\!\!\int_{\Omega}
			(-\partial_t\psi)\,T_m(\omega_i)\,dx\,dt
			+
			d_i \int_{0}^{T}\!\!\int_{\Omega}
			\nabla\psi\,
			T_m'(\omega_i)\,
			\nabla\omega_i\,dx\,dt
			\\
			&\qquad \ge
			\int_{0}^{T}\!\!\int_{\Omega}
			\psi\,\Lambda^m_i\,dx\,dt
			-
			C\,m\,\eta^{1/2},
		\end{align*}
		where
		\[
		\omega_i
		=
		f_i + \eta \sum_{j\ne i} f_j,
		\ \text{and} \
		\Lambda^m_i
		=
		T_m'(\omega_i)
		\left(
		Q_i + \eta \sum_{j\ne i} Q_j
		\right).
		\]
	\end{proof}
	Next, we take the limit $\eta \to 0$ to recover the inequality satisfied by each $f_i$ 
	for $i \in \mathbb{N}$. It is easy to see that the last term vanishes in the limit, 
	yielding a weak supersolution of the nonlinear fragmentation equation \eqref{NFE}.

	\section{Weak supersolution to the nonlinear fragmentation equation \eqref{NFE}} \label{EWSNFE}
	
	In this section, we focus on constructing a weak supersolution to the nonlinear
	fragmentation equation \eqref{NFE}. We will consider the limits $\eta \to 0$ and
	$m \to \infty$. We begin with the following proposition.
	\begin{prop}\label{eta to zero}
		Let all the assumption as in Proposition \ref{epsilon compactness in the inequation} is satisfied. The limit $\eta\to 0$ in the estimate \eqref{control estimate on epsilon} yields: for every
		nonnegative test function $\psi \in C_c^\infty([0,T)\times\Omega)$,
		\begin{align}\label{control on eta}
			\int_\Omega \psi(0,x)\,T_m(f_i^{\mathrm{in}})\,dx
			+&
			\int_0^T\!\!\int_\Omega
			(-\partial_t\psi)\,T_m(f_i)\,dx\,dt
			+\, d_i \int_0^T\!\!\int_\Omega
			\nabla\psi \cdot  T_m'(f_i)\,\nabla f_i
			\,dx\,dt \nonumber
			\\
			&\ge
			\int_0^T\!\!\int_\Omega
			\psi\,T_m'(f_i)\,Q_i
			\,dx\,dt.
		\end{align}
		Here $Q_i$ is as defined in \eqref{NFE Operator}.
	\end{prop}
	\begin{proof}
		We now let $\eta \to 0$ in the estimate \eqref{control estimate on epsilon}. We divide the proof in several steps.
		\newline 
		\textbf{Uniform estimates.}
		Thanks to the Fatou's lemma and the estimate \eqref{gradient estimate:regularization}, we have that
		\begin{equation}\label{eq:eta_mass}
			\left \{
			\begin{aligned}
				\sum_{i=1}^{\infty} \| f_i \|_{L^1((0,T)\times\Omega)}
				&\lesssim
				C,\\
				\sum_{i=1}^{\infty} \| \nabla f_i \|_{L^1((0,T)\times\Omega)}
				&\lesssim
				C.
			\end{aligned}
			\right .
		\end{equation}
		Moreover, Lemma \eqref{bound on NR operator} yields
		\begin{equation}\label{eq:eta_Qeps}
			\sum_{i=1}^{\infty} \| Q_{i,\varepsilon} \|_{L^1((0,T)\times\Omega)}
			\lesssim C.
		\end{equation}
		By Fatou's lemma, for every fixed $M\in\mathbb{N}$,
		\begin{equation}\label{eq:eta_Q}
			\sum_{i=1}^{M} \| Q_i \|_{L^1((0,T)\times\Omega)}
			\le
			\liminf_{\varepsilon\to0}
			\sum_{i=1}^{M} \| Q_{i,\varepsilon} \|_{L^1((0,T)\times\Omega)}
			\lesssim
			C.
		\end{equation}
		
		\medskip
		
		\noindent\textbf{Convergence of $\omega_i$.}
		Recall that
		\[
		\omega_i
		=
		f_i + \eta \sum_{j\neq i} f_j.
		\]
		Using \eqref{eq:eta_mass}, we estimate
		\begin{align*}
			\int_0^T\!\!\int_\Omega
			\left|
			\omega_i - f_i
			\right|
			\,dx\,dt
			&=
			\eta \int_0^T\!\!\int_\Omega
			\left|
			\sum_{j\neq i} f_j
			\right|
			dx\,dt
			\\
			&\le
			\eta \sum_{j\neq i}
			\| f_j \|_{L^1((0,T)\times\Omega)}
			\;\xrightarrow[\eta\to0]{}\;0.
		\end{align*}
		Hence,
		\begin{equation}\label{eq:eta_omega}
			\omega_i \longrightarrow f_i
			\quad\text{in }L^1((0,T)\times\Omega)
			\quad\text{and a.e. in }(0,T)\times\Omega.
		\end{equation}
		
		\medskip
		
		\noindent\textbf{Reaction term.}
		Using \eqref{eq:eta_Q}, we have
		\begin{align*}
			\int_0^T\!\!\int_\Omega
			\left|
			Q_i + \eta \sum_{j\neq i} Q_j - Q_i
			\right|
			\,dx\,dt
			&\le
			\eta \sum_{j\neq i}
			\| Q_j \|_{L^1((0,T)\times\Omega)}
			\\
			&\xrightarrow[\eta\to0]{} 0.
		\end{align*}
		Since $0 \le T_m' \le 1$, by dominated convergence and \eqref{eq:eta_omega},
		\begin{equation}\label{eq:eta_reaction}
			T_m'(\omega_i)
			\left(
			Q_i + \eta \sum_{j\neq i} Q_j
			\right)
			\longrightarrow
			T_m'(f_i)\,Q_i
			\quad\text{in }L^1((0,T)\times\Omega).
		\end{equation}
		
		\medskip
		
		\noindent\textbf{Diffusion term.}
		We write
		\[
		\nabla \omega_i
		=
		\nabla f_i + \eta \sum_{j\neq i} \nabla f_j.
		\]
		Thanks to the estimate \eqref{eq:eta_mass}, we obtain
		\begin{align*}
			\int_0^T\!\!\int_\Omega
			\left|
			\nabla \omega_i - \nabla f_i
			\right|
			\,dx\,dt
			&\le
			\eta \sum_{j\neq i}
			\| \nabla f_j \|_{L^1((0,T)\times\Omega)}
			\\
			&\xrightarrow[\eta\to0]{} 0.
		\end{align*}
		Together with \eqref{eq:eta_omega} and the boundedness of $T_m'$, this yields
		\begin{equation}\label{eq:eta_grad_term}
			T_m'(\omega_i)\,\nabla \omega_i
			\longrightarrow
			T_m'(f_i)\,\nabla f_i
			\quad\text{in }L^1((0,T)\times\Omega).
		\end{equation}
		
		\medskip
		
		\noindent\textbf{Vanishing last term.}
		The last term of the estimate \eqref{control estimate on epsilon} satisfies
		\[
		C\,m^2(1+T)^2\,\eta^{1/2}
		\;\xrightarrow[\eta\to0]{}\;0.
		\]
		
		\medskip
		
		\noindent\textbf{Limit weak formulation.}
		Letting $\eta \to 0$ in the estimate \eqref{control estimate on epsilon} and using
		\eqref{eq:eta_omega}–\eqref{eq:eta_grad_term}, we conclude that for every
		nonnegative test function $\psi \in C_c^\infty([0,T)\times\Omega)$,
		\begin{align*}
			\int_\Omega \psi(0,x)\,T_m(f_i^{\mathrm{in}})\,dx
			+&
			\int_0^T\!\!\int_\Omega
			(-\partial_t\psi)\,T_m(f_i)\,dx\,dt
			+\, d_i \int_0^T\!\!\int_\Omega
			\nabla\psi \cdot  T_m'(f_i)\,\nabla f_i
			\,dx\,dt
			\\
			&\ge
			\int_0^T\!\!\int_\Omega
			\psi\,T_m'(f_i)\,Q_i
			\,dx\,dt.
		\end{align*}
	\end{proof}
	In the next theorem we establish existence of global-in-time nonnegative weak supersolution to the nonlinear fragmentation system \eqref{NFE}.

	\begin{thm}\label{existence result_supersolution}
		Let the coagulation kernel $a_{i,j}$, the breakage kernel $b_{i,j}^k$, and the diffusion coefficients $d_i$ satisfy either Assumption~\ref{weaker assumption} or Assumption~\ref{strong assumption} for all $i,j,k \in \mathbb{N}$. Suppose that the initial data satisfy
		\[
		\sum_{i=1}^{\infty} \frac{1}{\sqrt{d_i}} \, \|f_i^{\rm{in}}\|_{L^1(\Omega)}^{\frac12} < +\infty \ \mbox{and} \ \left\|\sum_{i=1}^{\infty}if_i^{\rm{in}}\right\|_{L^2(\Omega)}<+\infty.
		\]
		Then the system~\eqref{NFE} admits a global-in-time nonnegative weak supersolution $\{f_i\}_{i\in\mathbb{N}}$ such that, for every $i \in \mathbb{N}$,
		\[
		f_i \in L^1\!\left((0,T); W^{1,1}(\Omega)\right)
		\cap L^2\!\left((0,T)\times\Omega\right).
		\]
	\end{thm}
	\begin{proof}
		Passing to the limit with respect to the level-set parameter $``m"$ yields the desired
		result. We divide the proof into a few steps.
		\newline
		\textbf{Reaction term.}
		We consider
		\[
		\int_0^T\!\!\int_\Omega
		\psi\,T_m'(f_i)\,Q_i \,dx\,dt.
		\]
		By construction of the truncation, $0 \le T_m'(r) \le 1$ for all $r\ge0$, and
		\[
		T_m'(f_i) \longrightarrow 1
		\quad \text{a.e. in } (0,T)\times\Omega
		\quad \text{as } m \to \infty.
		\]
		Moreover,
		\[
		\big|\psi\,T_m'(f_i)\,Q_i\big|
		\le
		\|\psi\|_{L^\infty}\,|Q_i|,
		\]
		with $Q_i \in L^1((0,T)\times\Omega)$.  
		By the dominated convergence theorem,
		\[
		\int_0^T\!\!\int_\Omega
		\psi\,T_m'(f_i)\,Q_i \,dx\,dt
		\;\xrightarrow[m\to\infty]{}\;
		\int_0^T\!\!\int_\Omega
		\psi\,Q_i \,dx\,dt.
		\]
		
		\medskip
		
		\noindent\textbf{Diffusion term.}
		We next consider
		\[
		d_i \int_0^T\!\!\int_\Omega
		\nabla\psi \cdot T_m'(f_i)\,\nabla f_i \,dx\,dt.
		\]
		Since $T_m'(f_i)\to1$ a.e. and
		\[
		\big|\nabla\psi \cdot T_m'(f_i)\,\nabla f_i\big|
		\le
		\|\nabla\psi\|_{L^\infty} |\nabla f_i|,
		\]
		with $\nabla f_i \in L^1((0,T)\times\Omega)$. Hence, dominated convergence theorem yields
		\[
		d_i \int_0^T\!\!\int_\Omega
		\nabla\psi \cdot T_m'(f_i)\,\nabla f_i \,dx\,dt
		\;\xrightarrow[m\to\infty]{}\;
		d_i \int_0^T\!\!\int_\Omega
		\nabla\psi \cdot \nabla f_i \,dx\,dt.
		\]
		
		\medskip
		
		\noindent\textbf{Time derivative term.}
		Since $T_m(r)=r$ for $0\le r\le m$ and $T_m(r)\to r$ for all $r\ge0$, we have
		\[
		T_m(f_i) \longrightarrow f_i
		\quad \text{a.e. in } (0,T)\times\Omega.
		\]
		Moreover,
		\[
		|T_m(f_i)| \le f_i,
		\]
		with $f_i \in L^1((0,T)\times\Omega)$. Hence,
		\[
		\int_0^T\!\!\int_\Omega
		(-\partial_t\psi)\,T_m(f_i)\,dx\,dt
		\;\xrightarrow[m\to\infty]{}\;
		\int_0^T\!\!\int_\Omega
		(-\partial_t\psi)\,f_i\,dx\,dt.
		\]
		
		\medskip
		
		\noindent\textbf{Initial term.}
		Similarly,
		\[
		\int_\Omega \psi(0,x)\,T_m(f_i^{\mathrm{in}})\,dx
		\;\xrightarrow[m\to\infty]{}\;
		\int_\Omega \psi(0,x)\,f_i^{\mathrm{in}}\,dx.
		\]
		
		\medskip
		
		\noindent\textbf{Limit weak formulation.}
		Letting $m\to\infty$, we conclude that for every nonnegative
		$\psi \in C_c^\infty([0,T)\times\Omega)$,
		\begin{align*}
			\int_\Omega \psi(0,x)\,f_i^{\mathrm{in}}(x)\,dx
			+&
			\int_0^T\!\!\int_\Omega
			(-\partial_t\psi)\,f_i \,dx\,dt
			+
			d_i \int_0^T\!\!\int_\Omega
			\nabla\psi \cdot \nabla f_i \,dx\,dt
			\\
			&\ge
			\int_0^T\!\!\int_\Omega
			\psi\,Q_i \,dx\,dt,
		\end{align*}
		where
		\[
		Q_i
		=
		\sum_{j=i+1}^{\infty}\sum_{k=1}^{j-1}
		b^{\,i}_{\,j-k,k}\,a_{\,j-k,k}\,f_{j-k}f_k
		-
		\sum_{j=1}^{\infty} a_{i j}\,f_i f_j.
		\]
	\end{proof}
	In the next section, we devote ourselves to establish the existence of a
	nonnegative global-in-time weak distributional solution to the nonlinear
	fragmentation system~\eqref{NFE}.

	\section{Proof of Theorem \ref{existence result}}\label{final_section}
	
	Assume that there exists a nonnegative test function
	\(
	\psi \in C_c^\infty([0,T)\times \Omega)
	\)
	such that, for some $i\in\mathbb{N}$,
	\begin{align}
		\int_\Omega \psi(0,x)\, f_i^{\mathrm{in}}(x)\,dx
		&+
		\int_0^T \int_\Omega
		\left(-\partial_t \psi\right) f_i
		\,dx\,dt
		+
		\int_0^T \int_\Omega
		\nabla \psi \cdot \nabla f_i
		\,dx\,dt
		\nonumber\\
		&>\;
		\int_0^T \int_\Omega
		\psi\, Q^i
		\,dx\,dt .
		\label{ineq:strict}
	\end{align}
	
	Multiplying \eqref{ineq:strict} by $i$ and summing over $i\ge 1$, we obtain
	\begin{align}\label{eq: strict positivity}
		\int_\Omega
		\psi(0,x)\,
		\sum_{i=1}^\infty i f_i^{\mathrm{in}}(x)\,dx
		&+
		\int_0^T \int_\Omega
		\left(-\partial_t \psi\right)
		\sum_{i=1}^\infty i f_i
		\,dx\,dt \nonumber\\
		&+
		\int_0^T \int_\Omega
		\nabla \psi \cdot
		\sum_{i=1}^\infty i \nabla f_i
		\,dx\,dt
		\;>\; 0 .
	\end{align}
	By Proposition~\ref{prop:infinite_mass}, there exists a strong solution
	$\{f_{i,\varepsilon}\}_{i\ge 1}$ to
	\eqref{eq:RNFE}–\eqref{eq:RNFEoperator}.
	In particular, $f_{i,\varepsilon}$ is a weak solution and satisfies
	\begin{align*}
		\int_\Omega
		\psi(0,x)\, f_{i,\varepsilon}^{\mathrm{in}}(x)\,dx
		+
		\int_0^T \int_\Omega
		\left(-\partial_t \psi\right)
		f_{i,\varepsilon}
		\,dx\,dt
		+
		\int_0^T \int_\Omega
		\nabla \psi \cdot
		\nabla f_{i,\varepsilon}
		\,dx\,dt
		=
		\int_0^T \int_\Omega
		\psi\, Q_\varepsilon^i
		\,dx\,dt .
	\end{align*}
	Multiplying the above identity by $i$ and summing over $i\ge 1$, we obtain
	\begin{align*}
		\int_\Omega
		\psi(0,x)\,
		\sum_{i=1}^{\infty} i f_{i,\varepsilon}^{\mathrm{in}}(x)\,dx
		+
		\int_0^T \int_\Omega
		\left(-\partial_t \psi\right)
		\sum_{i=1}^{\infty} i f_{i,\varepsilon}
		\,dx\,dt +
		\int_0^T \int_\Omega
		\nabla \psi \cdot
		\sum_{i=1}^{\infty} i \nabla f_{i,\varepsilon}
		\,dx\,dt
		= 0 .
	\end{align*}
	Here we use the fact that $f_i$ is the supersolution of the system \eqref{NFE} (see Theorem \ref{existence result_supersolution}). Since the right-hand side contains no nonlinear terms, we can pass to the limit
	$\varepsilon \to 0$ using the compactness results established in Section \ref{UECRS}.
	This yields
	\begin{align*}
		\int_\Omega
		\psi(0,x)\,
		\sum_{i=1}^{\infty} i f_i^{\mathrm{in}}(x)\,dx
		&+
		\int_0^T \int_\Omega
		\left(-\partial_t \psi\right)
		\sum_{i=1}^{\infty} i f_i
		\,dx\,dt \\
		&+
		\int_0^T \int_\Omega
		\nabla \psi \cdot
		\sum_{i=1}^{\infty} i \nabla f_i
		\,dx\,dt
		= 0 ,
	\end{align*}
	which contradicts \eqref{eq: strict positivity}. Therefore, for all nonnegative
	$\psi \in C_c^\infty([0,T)\times \Omega)$ and all $i\ge 1$,
	\begin{align*}
		\int_\Omega
		\psi(0,x)\, f_i^{\mathrm{in}}(x)\,dx
		+
		\int_0^T \int_\Omega
		\left(-\partial_t \psi\right) f_i
		\,dx\,dt
		+
		\int_0^T \int_\Omega
		\nabla \psi \cdot \nabla f_i
		\,dx\,dt
		\;\le\;
		\int_0^T \int_\Omega
		\psi\, Q^i
		\,dx\,dt .
	\end{align*}
	This shows that $\{f_i\}_{i\ge 1}$ is a subsolution to the nonlinear fragmentation equation \eqref{NFE}.
	On the other hand, Theorem \ref{existence result_supersolution} guarantees the same $\{f_i\}_{i\ge 1}$ is also a supersolution to \eqref{NFE}.
	Hence, the existence of both a subsolution and a supersolution ensures the existence of a weak solution
	in the sense of Definition~\ref{definition weak solution}.
	
	\medskip
	
	\textbf{Acknowledgment:} The first author received a Postdoctoral Research Fellowship from the Department of Atomic Energy (DAE), Government of India. The second author gratefully acknowledges support from the Anusandhan National Research Foundation (ANRF), India, through the National Post-Doctoral Fellowship (NPDF) [File No. PDF/2025/007779] and the Postdoctoral Research Fellowship he received during the postdoctoral tenure at Harish-Chandra Research Institute (HRI) from the Department of Atomic Energy (DAE), Government of India.
	\noindent
	
	\medskip
	\bibliography{Refs.bib}
	\bibliographystyle{abbrv}
\end{document}